\documentclass[10pt]{amsart}
\usepackage[margin=2cm]{geometry}
\usepackage{graphicx}
\usepackage{amsmath,amsthm}
\usepackage{amsfonts,amssymb}
\usepackage{amscd}
\usepackage{amsthm}
\usepackage{verbatim}
\usepackage[cp1251]{inputenc}
\usepackage[english]{babel}
 \usepackage{tikz}
\usepackage[all,cmtip]{xy}
\usepackage{tikz-cd}
\usepackage{makecell,xcolor}
\usepackage{float}
\usepackage{subcaption} 
\usepackage{fancyhdr}
\theoremstyle{plain}

\usetikzlibrary{matrix,arrows,decorations.pathmorphing}

\tikzstyle{slice}=[text=gray]

\usepackage{xcolor}

\usepackage{enumitem}
\newtheorem{lemma}{Lemma}
\newtheorem{propos}{Proposition}
\newtheorem{theorem}{Theorem}
\newtheorem{corollary}{Corollary}

\theoremstyle{remark}
\newtheorem*{remark}{Remark}
\newtheorem*{theorem*}{Theorem}

\theoremstyle{definition}
\newtheorem{definition}{Definition}[section]

\newcommand{\Ext}{\mathop{\mathrm{Ext}}\nolimits}
\newcommand{\End}{\mathop{\mathrm{End}}\nolimits}
 \def\Z{\mathbb Z}

\newcommand{\modl}{\mathop{\mathrm{mod}}\nolimits}

\newcommand{\Hom}{\mathop{\mathrm{Hom}}\nolimits}

\newcommand{\rs}{\mathop{\mathrm{rsupp}}\nolimits}
\newcommand{\ls}{\mathop{\mathrm{lsupp}}\nolimits}
\newcommand{\RHom}{\mathop{\mathrm{RHom}}\nolimits}

\newcommand{\im}{\mathop{\mathrm{Im}}\nolimits}
\newcommand{\Aut}{\mathop{\mathrm{Aut}}\nolimits}
\newcommand{\Ker}{\mathop{\mathrm{Ker}}\nolimits}

\def\kk{\mathrm{k}}
\def\ve{\varepsilon}

\newtheorem*{genericthm*}{\thistheoremname}
\newenvironment{namedthm*}[1]
  {\renewcommand{\thistheoremname}{#1}%
   \begin{genericthm*}}
  {\end{genericthm*}}

\makeatletter
\renewcommand{\email}[2][]{%
  \ifx\emails\@empty\relax\else{\g@addto@macro\emails{,\space}}\fi%
  \@ifnotempty{#1}{\g@addto@macro\emails{\textrm{(#1)}\space}}%
  \g@addto@macro\emails{#2}%
}
\makeatother

\def\@evenhead{\hfil\sc A. Nordskova, Y. Volkov \hfil} 
\def\@oddhead{\hfil\sc  Faithful actions of braid groups by spherical twists \hfil} 
\def\@oddfoot{\hfil\sc \thepage \hfil} 
\def\@evenfoot{\hfil\sc \thepage \hfil} 

\def\ps@firstpage{\ps@plain
  \def\@oddfoot{\hfil\sc \thepage \hfil}%

}
\makeatother

\author{Anya Nordskova}
\address{anya.nordskova@gmail.com, L. Euler International Mathematical Institute, Saint Petersburg State University, 14th Line V.O., 29B, St. Petersburg 199178, Russia. }
\author{Yury Volkov}
\address{wolf86\_666@list.ru, Department of Mathematics and Computer Science, Saint Petersburg State University, Universitetskaya nab. 7-9, St. Peterburg, Russia.}

\thanks{\emph{ORCID}: 0000-0003-3592-4273 (first author), 0000-0003-0704-7912 (second author).}

\usepackage{lipsum}

\begin{document}
\title{Faithful actions of braid groups by twists along ADE-configurations of spherical objects}
\maketitle

\begin{abstract}
    We prove that the action of a generalized braid group on an enhanced triangulated category, generated by spherical twist functors along an ADE-configuration of $\omega$-spherical objects, is faithful for any $\omega \in \Z$, $\omega \neq 1$. \end{abstract}

\section{Introduction}

\let\thefootnote\relax\footnotetext{The work was supported by RFBR according to the research project 18-31-20004 and by the President's Program ``Support of Young Russian Scientists'' (project number MK-2262.2019.1). The first author was also supported by ``Native Towns'', a social investment program of PJSC ``Gazprom neft''. The second author was in part supported by Young Russian Mathematics award.\\ \emph{Key words:} Artin groups, spherical twists, triangulated categories, derived Picard groups \\ \emph{2010 Mathematics Subject Classification:} primary 18E30, 20F36, secondary 16E35, 14F05}

The notion of a spherical twist functor was first introduced in \cite{ST} by P. Seidel and R. Thomas in connection with the Kontsevich homological mirror symmetry program. Their original motivation was to look at the autoequivalences of the derived category of coherent sheaves on a variety that would arise as counterparts of generalized Dehn twists via mirror symmetry. In a general setting, a spherical twist is a functor constructed from a spherical object, an object of a triangulated category whose Ext algebra is the same as the cohomology of a sphere.

Spherical twists have proved to be a useful tool in algebraic geometry, representation theory and beyond. In particular, they appear in the contexts of categorifications of Braid groups \cite{R}, Bridgeland stability conditions manifolds \cite{Br}, \cite{HMS} and derived Picard groups \cite{VZ1}. There are several generalizations as well. For instance, in \cite{AL} the notion of a spherical twist along a spherical functor was introduced and in \cite{Seg} it was established that any triangulated autoequivalence is in fact a twist along some spherical functor. The theory of spherical twists constructed from spherical sequences was developed in \cite{Ef}. Groups that can be generated by two spherical twists constructed from spherical sequences were described in \cite{Volk} by the second named author.
 
Let $\Gamma$ be a simply-laced Dynkin diagram. Seidel and Thomas showed that spherical twists along a so-called $\Gamma$-configuration of $\omega$-spherical objects satisfy braid relations of type $\Gamma$ modulo natural isomorphisms, hence induce an action of an Artin group (generalized braid group) $B_\Gamma$ on the triangulated category in question. In the same paper they showed that for $\omega \geq 2$ and $\Gamma = A_n$ this action is faithful. By \cite{Ef} their result can be extended to the case $\omega = 0$. Seidel and Thomas also provided an example when the action is not faithful for $\omega = 1$. Later several more results on the faithfulness appeared. In \cite{BT} C. Brav and H. Thomas proved that the braid group action is faithful for $\omega = 2$ and all $\Gamma$. Recently Y. Qiu and J. Woolf generalised their result to $\omega \geq 2$ for the derived category of a Ginzburg dg-algebra \cite{QW}, which by the intrinsic formality result of A. Hochenegger and A. Krug \cite{HK} may be extended to any algebraic triangulated category, provided that $\omega \geq 4$.

 \bigskip

The main result of our paper is as follows: 
\bigskip

{\bf Theorem 1.} {\it Let $\Gamma$ be a simply-laced Dynkin diagram, $\omega \in \Z$, $\omega \neq 1$ and $\{P_i\}_{i \in \Gamma_0}$ a $\Gamma-$configuration of $\omega-$spherical objects in an enhanced triangulated category $\mathfrak{D}$. Then the action of $B_\Gamma$ on $\mathfrak{D}$ generated by the spherical twists $t_{P_i}$ is faithful. }

\bigskip

The new method we present enables us to prove that the braid group action is faithful for any enhanced triangulated category, all $\Gamma = A_n, D_n, E_6, E_7, E_8$ and all $\omega \neq 1$, almost simultaneously in all cases. Hence, Theorem 1 generalizes all of the existing faithfulness results in this setting. In particular, it covers the case $\omega < 0$ not considered before. Originally, Seidel and Thomas did not define $\omega-$spherical objects with negative $\omega$, but they are also worth considering (for instance, see \cite{HJ}, \cite{CS1}, \cite{CSP}, \cite{CS2}), despite being somewhat more exotic than those of non-negative CY dimension. Among non-positive dimensions, the case $\omega = 0$, which we originally aimed for, having a particular application to derived Picard groups of algebras in mind, happens to be the hardest to tackle.   

The proof we present is based on the following general idea. Let $\Aut(\mathfrak{D})$ be the group of autoequivalences of $\mathfrak{D}$ modulo natural isomorphisms. Let $t_{\alpha}$ be the image of a word $\alpha \neq 1$ in the generalized braid monoid $B_\Gamma^+$ under the homomorphism that sends each generator $s_i$ of $B_\Gamma^+$ to the corresponding spherical twist $t_{P_i}$. Then we show that there is a subset $I$ of the set of vertices of $\Gamma$, depending only on the integer $\omega$ and the object $t_{\alpha}(\bigoplus_{i=1}^n P_i)$, such that for every $j \in I$ the word $\alpha$ is left-divisible by $s_j$ in $B_\Gamma^+$ (see Lemma \ref{princ}  for the precise statement). In other words, for each $\alpha \neq 1$ we learn to recover at least one of its leftmost letters using only information about the object $t_{\alpha}(\bigoplus_{i=1}^n P_i)$. Roughly the same idea already manifested itself in the proof of Brav and Thomas for the case $\omega = 2$ in \cite{BT}. To prove the reconstruction statement in the case $\omega \geq 2$ we simplify and then extend their approach. In particular, our proof does not rely on the existence of the Garside structure.

On the other hand, when $\omega \leq 0$ obtaining the required reconstruction statement turns out to be a much more difficult task. Arguing by contradiction, in this case we inductively construct a presentation of $\alpha$ of a particular form, extracting data from the object $t_{\alpha}(\bigoplus_{i=1}^n P_i)$ (we refer to this process as 'factorization'). This step requires us to develop the theory of so-called two-term objects in a triangulated category with a configuration of spherical objects. When the desired presentation is constructed, the problem that remains lies entirely in the realm of generalized braid groups combinatorics. Namely, we show that applying suitable braid and commutation relations to our presentation, one can indeed make any $s_j$ with $j \in I$ appear as leftmost factors of $\alpha$, as promised in the reconstruction lemma (this step is referred to as 'braiding'). To this end we introduce a way of interpreting words in a braid monoid as decorated sets of vertices of the translation quiver $\Z \Gamma$ of $\Gamma$. In such terms we can formulate and prove a sufficient condition for when a word in a braid monoid is left-divisible by a generator $s_j$, which in particular finishes our proof. The condition itself and the approach we use in this part of the proof are rather general and might be of independent interest.

The structure of the paper is as follows. Section 2 introduces basic definitions and notations. Section 3 contains the main result of the paper as well as the key 'reconstruction lemma' on which its proof is based. The rest of the paper is mainly devoted to proving the key lemma. Section 4 provides a proof in the case $\omega \geq 2$. In Section 5 we outline the proof of the key lemma in the case $\omega \leq 0$ and then provide the details in the next three sections. In the last section one application of the main result of this paper is presented. Namely, we show how it may be used to make the first yet crucial step towards the description of the derived Picard groups of representation-finite selfinjective algebras.  

{\bf Acknowledgements.} We would like to thank Alexandra Zvonareva for inspiring discussions and useful remarks.

\section{Preliminaries}

Throughout this paper $\mathfrak{D}$ is a triangulated category linear over a field $\kk$ and with a fixed enhancement. For example, it can be an algebraic triangulated category in the sense of Keller (see \cite{Kel}), i.e. the stable category of some Frobenius category. In this case $\mathfrak{D}$ is equipped with functorial cones of natural transformations of exact functors and for every object $X$ of $\mathfrak{D}$. Let $D(\kk)$ denote the unbounded derived category of $\kk$-vector spaces. Then there is the derived Hom-complex functor $\RHom(X, -) \colon \mathfrak{D} \to D(\kk)$ and its right adjoint $ - \otimes X \colon D(\kk) \to \mathfrak{D}$. Denote by $\eta$ and $\ve$ respectively the unit and the counit of this adjunction.  

\bigskip Let $\Hom^k(A,B)$ denote $\Hom_{\mathfrak{D}}(A,B[k])$ and $\Hom^*(A,B) = \bigoplus_{k \in \Z} \Hom^k(A,B)$. For elements $g\in \Hom^k(A,B)$ and $f\in \Hom^l(B,C)$, we will write $f\circ g$ or simply $fg$ for the element ${f[k]\circ g\in \Hom^{k+l}(A,C)}$.

\begin{definition}(Seidel, Thomas \cite{ST}) Let $\omega \in \Z$. An object $P \in \mathfrak{D}$ is called {\it $\omega$-spherical} if

 \begin{enumerate}
 \item[(i)] $\dim_{\kk} \Hom^*(X,P) < \infty$ for any object $X \in \mathfrak{D}$.
 \item[(ii)] $\Hom^*(P,P) \cong k[t]/(t^2)$ as graded $\kk-$algebras, where $deg(t) = \omega$.

 \item[(iii)] $$\Hom^*(P,X) \times \Hom^*(X,P) \xrightarrow{\circ} \Hom^{*}(P,P)/{\langle Id_P \rangle}  \cong \kk$$ 
 is a perfect pairing for any $X \in \mathfrak{D}$.

 \end{enumerate}
 

 \end{definition}

Fix some undirected graph $\Gamma$. We will assume that $\Gamma$ is an ADE Dynkin diagram, but part of our arguments can be carried over to more general cases. Following Brav and Thomas \cite{BT}, we now define a $\Gamma$-configuration of spherical objects.

\begin{definition} A collection of $\omega$-spherical objects $\{P_i\}_{i\in \Gamma_0}$ enumerated by vertices of $\Gamma$ is a {\it $\Gamma-$configuration} if for any $i\neq j$

\begin{enumerate}
\item $\Hom^*(P_i, P_j)$ is one-dimensional if $i$ and $j$ are adjacent in $\Gamma$;
\item  $\Hom^*(P_i, P_j)=0$ if $i$ and $j$ are not adjacent in $\Gamma$.
\end{enumerate} \end{definition}

Now we represent the set $\Gamma_0$ of vertices of $\Gamma$ as a disjoint union of two sets $V^0$ and $V^1$ in such a way that each edge of $\Gamma$ has one endpoint in $V^0$ and the other one in $V^1$. Fix some integers $\omega_0$ and $\omega_1$ such that $\omega_0+\omega_1=\omega$. Assume additionally that $\omega_0,\omega_1\le 0$ if $\omega\le 0$ and $\omega_0,\omega_1 \geq 1$ if $\omega\ge 2$.
For example, we can simply take $\omega_0=\omega_1 = \frac{\omega}{2}$ if $\omega$ is even and $\omega_0 = \frac{\omega + 1}{2}$, $\omega_1 = \frac{\omega - 1}{2}$ if  $\omega$ is odd. 
It follows from the definition of an $\omega$-spherical object that we can shift the objects $\{P_i\}_{i \in \Gamma_0}$ of a $\Gamma$-configuration and assume without loss of generality that $\Hom^*(P_i, P_j)$ is concentrated in degree $\omega_u$ whenever $i$ belongs to $V^u$.  The index $u$ in the notations $\omega_u$ and $V^u$ will be always taken modulo $2$.

\begin{definition}\label{twidef}(Seidel, Thomas \cite{ST}) Let $P$ be an $\omega$-spherical object. The {\it spherical twist} functor $t_P$ along $P$ is defined by

 $$t_P(X) = cone(P \otimes \RHom(P, X) \xrightarrow{\ve_X} X) $$      \end{definition}

\begin{remark} $t_P$ is indeed a functor since we have functorial cones of natural transformations of exact functors. \end{remark}

 \begin{definition} The Artin group (generalized braid group) $B_\Gamma$ of type $\Gamma$ is generated by $s_i, i \in \Gamma_0$ subject to  braid relations $s_is_js_i = s_js_is_j$ for $i,j$ adjacent in $\Gamma$ and commutation relations $s_is_j = s_js_i$ for $i,j$ not adjacent in $\Gamma$. The braid monoid $B_\Gamma^+$ is a monoid given by the same generators and relations. \end{definition}
 
 For $\alpha \in B_{\Gamma}^+$ we will denote by $l(\alpha)$ the length of $\alpha$ in terms of the standard generators $s_i$, $i \in \Gamma_0$. 
 
We are now going to recall some crucial facts about spherical twists and spherical objects (see \cite{ST}) that we will actively use throughout the paper.

\begin{enumerate}
\item If $P$ is spherical, then $t_P$ is an autoequivalence of $\mathfrak{D}$ with a quasi inverse $t_P'$ defined by
$$t_P'(X) = cone(X\xrightarrow{\eta_X}P \otimes \RHom(P, X)^* )[-1],$$
where $^*$ denotes the dual with respect to $\kk$.
\item If $P$ is $\omega$-spherical, then $t_P(P) = P[1-\omega]$.
\item For $\{P_i\}_{i \in \Gamma_0}$ forming a $\Gamma$-configuration, the functors $t_{P_i}$ satisfy braid relations of type $\Gamma$ up to a natural isomorphism. In other words, there is a group homomorphism

$$F\colon B_\Gamma \xrightarrow{} \Aut(\mathfrak{D})$$

where $\Aut(\mathfrak{D})$ is the group of autoequivalences of $\mathfrak{D}$ modulo natural isomorphisms. For $\alpha \in B_\Gamma$, we denote $F(\alpha)$ by $t_\alpha$.
\item  $t_P(Q) = Q$ for spherical $P, Q$ not adjacent in their $\Gamma$-configuration (equivalently $\Hom^*(P,Q) = 0$). 
\end{enumerate}

\section{Main result}
We have just defined an action of the braid group $B_\Gamma$ of type $\Gamma$ on the category $\mathfrak{D}$. The main result of this paper says that this action is faithful for every integer $\omega \neq 1$. The rest of this paper is mainly devoted to proving this claim:

\begin{theorem}\label{main_thm} Let $\Gamma$ be a simply-laced Dynkin diagram, $\omega \in \Z$, $\omega \neq 1$ and $\{P_i\}_{i \in \Gamma_0}$ a $\Gamma-$configuration of $\omega-$spherical objects in an enhanced triangulated category $\mathfrak{D}$. Then the action of $B_\Gamma$ on $\mathfrak{D}$ generated by the spherical twists $t_{P_i}$ is faithful. \end{theorem}

Recall that for $\omega = 1$ there are examples when the action is not faithful (see \cite{ST}). 
 
 \bigskip

We will follow the same general strategy to prove Theorem \ref{main_thm} in the two cases $\omega\le 0$ and $\omega\ge 2$. Nevertheless, the case $\omega \geq 2$ turns out to be much easier and will be finished by the end of the next section. 

\bigskip 

By $N(A)$ for $A \subseteq \Gamma_0$ we denote the set of all neighbors of all vertices in $A$. Since the case $|\Gamma_0|=1$ is clear, we will assume hereafter that $\Gamma$ has at least two vertices. We will use the following auxiliary notation in our proof.
For each $i\in V^u$ and $j\in N(i)$, we fix some generator of $\Hom^{\omega_u}(P_i,P_j)$ and denote it by $\gamma_{i,j}$.
We also introduce the morphisms $\rho_{i,j} \colon P_j\rightarrow t_iP_j$ and $\xi_{i,j} \colon t_iP_j\rightarrow P_i[1-\omega_u]$ via the triangle from the Definition \ref{twidef}
$$
P_i[-\omega_u]\xrightarrow{\gamma_{i,j}[-\omega_u]}P_j\xrightarrow{\rho_{i,j}}t_iP_j\xrightarrow{\xi_{i,j}}P_i[1-\omega_u].
$$


   Let $\Lambda = \bigoplus\limits_{i=1}^n P_i$. We will also write $t_i$ instead of $t_{P_i}$ for brevity.

   We define {\it the minimal} and {\it the maximal nonzero degree} of an object in $\mathfrak{D}$:

 \begin{definition} Let $T \in \mathfrak{D}$. We say that $min(T)$ is {\it the minimal  nonzero degree} of $T$  if  
 $$ min(T) = \min\{k \in \Z | \Hom^{k+\omega}(\Lambda, T) \neq 0\}. $$
 
Analogously, $max(T)$ is {\it the maximal nonzero degree} of $T$ if 
 
 $$max(T) = \max \{ k\in \Z | \Hom^{k+\omega}(\Lambda, T) \neq 0 \}$$

Note that since $\Lambda$ is a direct sum of spherical objects, $min(T)$ and $max(T)$ always exist.

Since in general $T \in \mathfrak{D}$ is an object of an abstract triangulated category rather than a complex, we cannot speak of its components $T_r$ for $r \in \Z$, say, as of objects of some abelian category. We can, however, define what it means for a spherical object $P_j$ to be {\it a direct summand of $T_r$} (we will use the notation $T_r$ only as a part of this expression). 
\end{definition}

   \begin{definition}\label{dirsum} Let $T$ be an object of $\mathfrak{D}$. We say that $P_j$ with $j\in \Gamma_0$ is {\it a direct summand of $T_{r}$} if there exists a nonzero $f \in \Hom^{r+\omega}(P_j, T)$ such that $f\gamma_{k,j} = 0$ for any $k \in N(j)$. A morphism $f$ satisfying this condition will be referred to as {\it long}. In other words, a nonzero morphism $f\colon P_j\rightarrow T$ is called long if the induced morphism $\Hom^*(P_k,f) \colon \Hom^*(P_k,P_j)\rightarrow \Hom^*(P_k,T)$ is zero for any $k\not=j$. We also say that $P_j$ is {\it a direct summand of $T_{[a,b]}$} if $P_j$ is a direct summand of $T_{c}$ for some $c \in [a,b]$. \end{definition}

 \begin{remark} When $\omega = 0$ and $\mathfrak{D}$ is the bounded derived category of a basic finite-dimensional algebra $\Lambda$ (see Section 9), a long morphism $f \in \Hom^r(P_j,T)$ is simply a morphism induced by a socle morphism $P_j \xrightarrow{} P_j$. \end{remark}
   \begin{remark} Under some conditions, checking whether $P_i$ is a direct summand of $T_r$ is easier. More precisely, let $T \in \mathfrak{D}$ and $i\in V^u$.
   If ${\Hom^{r+\omega}(P_i,T)\not=0}$ for some $r\le min(T)-\omega_{u+1}-1$, then $P_i$ is a direct summand of $T_{r}$. Indeed, any composition of the form ${P_j\rightarrow P_i[\omega_{u+1}]\rightarrow T[r+\omega+\omega_{u+1}]}$ is zero by the definition of the minimal nonzero degree.
   Moreover, if ${\Hom^{min(T)+\omega}(P_j,T) = 0}$ for every $j \in N(i)$  and $\Hom^{min(T)+\omega-\omega_{u+1}}(P_i,T)\not=0$, then $P_i$ is a direct summand of $T_{min(T)-\omega_{u+1}}$, because in this case any composition of the form $P_j\rightarrow P_i[\omega_{u+1}]\rightarrow T[min(T)+\omega]$ is zero too. 
Analogously, if $\Hom^{r+\omega}(P_i,T)\not=0$ for some $i\in V^u$ and $r\ge max(T)-\omega_{u+1}+1$, then $P_i$ is a direct summand of $T_{r}$. 
\end{remark}

  We denote $t_\alpha(\Lambda)$ by $T_\alpha$. Let $min(\alpha)$ and $max(\alpha)$ denote the minimal and the maximal nonzero degrees of $T_{\alpha}$ respectively. Following \cite{BT}, we will deduce the faithfulness of the braid group action from the injectivity of the induced monoid homomorphism ${B^+_{\Gamma} \xrightarrow{} \Aut(\mathfrak{D})}$. In turn, to prove that the monoid homomorphism is injective, we require a tool that would allow us to recover a leftmost factor of a reduced expression of $\alpha\in B_\Gamma^+$ using only information about $T_\alpha$. This tool is presented in the following key lemma.

\bigskip

\begin{lemma}\label{princ} Let $\alpha \in B^+_\Gamma$, $\alpha \neq 1$. For $u \in \{0,1\}$ let $I_{u,\alpha}=[min(\alpha),min(\alpha)-\omega_{u+1}]$ in the case $\omega\le 0$ and $I_{u,\alpha}=[max(\alpha)-\omega_{u+1}+1,max(\alpha)]$ in the case $\omega\ge 2$. Then for any $u\in\{0,1\}$ and any $j\in V^u$ such that the corresponding object $P_j$ is a direct summand of $(T_\alpha)_{I_{u,\alpha}}$, the word $\alpha$ can be written as
$$\alpha = s_j \alpha'$$
for some $\alpha' \in B^+_\Gamma$ with $l(\alpha) = l(\alpha') + 1$.   \end{lemma}

  Before proving Lemma \ref{princ} we will to show how it easily implies our main result. Even before that, however, we are going to introduce and prove some rather technical facts regarding the way spherical twists affect the minimal and the maximal degree of an object as well as its direct summands in the sense of Definition \ref{dirsum}. These statements will be extensively used throughout the paper. Although here the two cases $\omega\le 0$ and $\omega \ge 2$ can be unified, we are going to consider them separately to avoid giving the impression that the case $\omega\ge 2$ is more complicated than it really is.

\begin{lemma}\label{tech0} Let $\omega\le 0$ and let $T$ be an object of $\mathfrak{D}$. Let $m$ be the minimal nonzero degree of $T$.
\begin{enumerate}
\item If $\Hom^r(P_i,t_i^{-1}T) \neq 0$, then $m\le r-1$.
\item For $k\in V^u$, $r\le m-\omega_{u+1}$ and $i\not=k$, $P_k$ is a direct summand of $(t_i^{-1}T)_{r}$ if and only if $P_k$ is a direct summand of $T_{r}$.
\item The minimal nonzero degree of $t_iT$ belongs to ${[m-1+\omega,m]}$. Moreover, $P_k$ can be a direct summand of $(t_iT)_{r}$ with $r<m$ only if $k=i$.
\end{enumerate}
\end{lemma}
\begin{proof} \begin{enumerate} 
\item As spherical twists are autoequivalences, we have $$0 \neq \Hom^r(P_i, t_i^{-1}T) \cong \Hom^r(t_iP_i, T)=  \Hom^r(P_i[1-\omega], T)\cong \Hom^{r-1+\omega}(P_i, T).$$
Hence, the minimal nonzero degree $m$ of $T$ is not greater than $r-1$. 

\item Suppose that $k \notin N(i)$. Then $t_iP_k=P_k$ and since $t_i$ is an autoequivalence, we have an isomorphism $t_i \colon \Hom^{r+\omega}(P_k, t_i^{-1}T)\cong \Hom^{r+\omega}(P_k, T)$. It is sufficient to show that $f\colon P_k\rightarrow t_i^{-1}T[r+\omega]$ is long if and only if $t_if$ is long. Pick some $l\in N(k)$. If $l\not\in N(i)$, then $t_iP_l=P_l$ and we have
$\im\Hom^*(P_l,t_if)=t_i\im\Hom^*(P_l,f)$. Since $t_i$ is an autoequivalence, we have $\Hom^*(P_l,t_if)=0$ if and only if $\Hom^*(P_l,f)=0$.

If $l\in N(i)$, then $i\in V^u$ and we have the following triangle
$$P_l\xrightarrow{\rho_{i,l}}t_iP_l\xrightarrow{\xi_{i,l}} P_i[1-\omega_u].$$
Since $\Hom^*(P_i,P_k)=0$, we have $\gamma_{l,k}=g\rho_{i,l}$ for some $g\in\Hom^{\omega_{u+1}}(t_iP_l,P_k)$. Hence, if $(t_if)\gamma_{l,k}\not=0$, then $f(t_i^{-1}g)\not=0$, and $\Hom^*(P_l,f)\not=0$. On the other hand, if $f\gamma_{l,k}\not=0$, then the composition 
$$P_l\xrightarrow{\rho_{i,l}}t_iP_l\xrightarrow{t_i(f\gamma_{l,k})} T[r+\omega+\omega_{u+1}]$$
 is nonzero, because ${r+\omega-1\le m+\omega_u-1<m}$. Thus, $${\Hom^{r+\omega+\omega_{u+1}}(P_i[1-\omega_u], T)\cong\Hom^{r+2\omega-1}(P_i, T)=0}.$$
 Then we have ${\Hom^*(P_l,t_if)\not=0}$. We conclude that indeed $f$ is long if and only if $t_if$ is long.

Now suppose that $k \in N(i)$.  Consider the triangle
\begin{equation}\label{tri01}
P_k \xrightarrow{\rho_{i,k}} t_iP_k \xrightarrow{\xi_{i,k}} P_i[1-\omega_{u+1}]
\end{equation}
Since the minimal nonzero degree of $T$ is $m$ and ${r+\omega_{u+1}-1 \leq m-1}$, we have $\Hom^{r+\omega}(P_i[1-\omega_{u+1}], T) = 0$.  Hence, for any nonzero $f \colon P_k\rightarrow t_i^{-1}T[r+\omega]$ one has $t_if\rho_{i,k}\not=0$. Since $(t_if)\rho_{i,k}\gamma_{i,k}[-\omega_{u+1}]=0$, it remains to show that $(t_if)\rho_{i,k}\gamma_{l,k}[-\omega_{u+1}]=0$ for $l \in N(k) \setminus \{i\}$
if $f$ is long.  Indeed, suppose it is nonzero and apply $t_i^{-1}$. Since $l \notin N(i)$, we get a nonzero morphism $P_l[-\omega_{u+1}]\xrightarrow{} P_k \xrightarrow{f} t_i^{-1}T[r+\omega]$, which contradicts $f$ being a long morphism.

Now pick a long morphism ${f' \colon P_k\rightarrow T[r+\omega]}$. Then ${f'\gamma_{i,k}[-\omega_{u+1}]=0}$ by the definition of a long morphism, and hence ${f'=(t_i f)\rho_{i,k}}$ for some ${f \colon P_k\rightarrow t_i^{-1}T[r+\omega]}$.
We have ${0=f\gamma_{i,k} \colon P_i \rightarrow t_i^{-1}T[r+\omega+\omega_{u+1}]}$, because otherwise the minimal nonzero degree of $T$ would be not greater than ${r+\omega+\omega_{u+1}-1\le m+\omega-1<m}$ by the first assertion of the current lemma. Pick some $l \in N(k) \setminus \{i\}$. 
Since $\Hom^*(P_l,P_i)=0$, the morphism $t_i\gamma_{l,k}:P_l\rightarrow t_iP_k[\omega_{u+1}]$ factors through $\rho_{i,k}[\omega_{u+1}]$. Therefore, $t_i\gamma_{l,k}$ equals $\rho_{i,k}\gamma_{l,k}$ modulo a nonzero scalar factor. Since $(t_i f)\rho_{i,k}$ is long, one has $(t_i f)\rho_{i,k}\gamma_{l,k}=0$. Applying $t_i^{-1}$, we see that $f\gamma_{l,k}=0$ as well. Thus, there exists a long morphism from $P_k$ to $T[r+\omega]$ if and only if there exists a long morphism from $P_k$ to $t_i^{-1}T[r+\omega]$. 

\item First we show that the minimal nonzero degree of $t_iT$ is not greater than $m$. There exists some $k\in\Gamma_0$ such that $P_k$ is a direct summand of $T_{m}$. If $k\not=i$ and the minimal nonzero degree of $t_iT$ is greater than $m$, then $P_k$ is a direct summand of $(t_iT)_{m}$ by the second assertion of this lemma and we get a contradiction. If $k=i$, then $\Hom^{m+\omega}(P_i,T)\not=0$ and the minimal nonzero degree of $t_iT$ is not greater than $m+\omega-1<m$ by the first assertion of this lemma. Thus, the minimal nonzero degree of $t_iT$ is indeed not greater than $m$.

Now suppose $P_k$ is a direct summand of $(t_iT)_{r}$ for $r<m$, $k \not=i$. If $k\notin N(i)$, then $0\neq \Hom^{r+\omega}(P_k, t_iT) \cong \Hom^{r+\omega}(P_k, T)$, which contradicts the minimal nonzero degree of $T$ being $m$. If $k \in N(i)$, then there exists a long morphism $ P_k \xrightarrow{f} t_iT[r+\omega]$, in particular, $f\gamma_{i,k}=0$. Consider the triangle
$$P_{i}[\omega_{u}-1] \xrightarrow{t_i^{-1}\gamma_{i,k}[-\omega_{u+1}]} t_i^{-1}P_k \xrightarrow{t_i^{-1}\rho_{i,k}} P_k. $$
Since $ t_i^{-1}(f\gamma_{i,k}[-\omega_{u+1}])=0$, the morphism $t_i^{-1}f$ factors through $P_k$. Thus, the minimal nonzero degree of $T$ is not greater than $r < m$, but this is impossible.

 Now it follows from what we have already established that if the minimal nonzero degree of $t_iT$ equals $d<m$, then ${\Hom^{d+\omega}(P_i,t_iT) \neq 0}$.
Applying $t_i^{-1}$, we get $\Hom^{d+1}(P_i,T) \neq 0$, and hence $d \ge  m-1+\omega$.
\end{enumerate}
\end{proof}

\begin{lemma}\label{tech2} Let $\omega\ge 2$, let $T$ be an object of $\mathfrak{D}$ and let $h$ be its maximal nonzero degree.
\begin{enumerate}
\item If $\Hom^r(P_i,t_i^{-1}T) \neq 0$, then  $h\ge r-1$.

\item Suppose that $k\in \Gamma_0\setminus\big(\{i\}\cup N(i)\big)$. Then $\Hom^{r+\omega}(P_k,t_i^{-1}T)\not=0$ if and only if $\Hom^{r+\omega}(P_k,T)\not=0$.

\item Suppose that $r\ge h-\omega_{u+1}+2$ and $k\in V^u$, $k\not=i$. If $\Hom^{r+\omega}(P_k,t_i^{-1}T)\not=0$, then $\Hom^{r+\omega}(P_k,T)\not=0$.

\item Suppose that $r\ge h-\omega_{u+1}+1$ and $k\in V^u$, $k\not=i$. If $\Hom^{r+\omega}(P_k,T)\not=0$, then $\Hom^{r+\omega}(P_k,t_i^{-1}T)\not=0$.

\item The maximal nonzero degree of $t_iT$ is not less than $h$.
\end{enumerate}
\end{lemma}
\begin{proof} \begin{enumerate}
\item Just as in Lemma \ref{tech0}, we get ${0 \neq \Hom^{r-1+\omega}(P_i, T)}$, and hence $h\ge r-1$. 

\item Since $k\not\in \{i\}\cup N(i)$, we have
$$ \Hom^{r+\omega}(P_k,t_i^{-1}T)=\Hom^{r+\omega}(t_i^{-1}P_k,t_i^{-1}T)\cong \Hom^{r+\omega}(P_k,T).$$

\item The case $k \not\in N(i)$ is already considered. Suppose now that  $k \in N(i)$.  Consider the triangle \eqref{tri01}.
For any nonzero ${f \colon P_k\rightarrow t_i^{-1}T[r+\omega]}$, one has ${0\not=(t_if)\rho_{i,k}:P_k\rightarrow T[r+\omega]}$, because ${r+\omega_{u+1}-1 \geq h+1}$. The required assertion now follows.

\item Due to the assertions already proven, it remains to consider the case  $k \in N(i)$. 
Consider the triangle \eqref{tri01} and
pick a nonzero morphism ${f' \colon P_k\rightarrow T[r+\omega]}$. Since ${r+\omega_{u+1}\ge h+1}$, one has ${f'\gamma_{i,k}[-\omega_{u+1}]=0}$, and hence ${f'=(t_i f)\rho_{i,k}}$ for some nonzero ${f \colon P_k\rightarrow t_i^{-1}T[r+\omega]}$. This finishes the proof.

\item There exists some $k\in\Gamma_0$ such that ${\Hom^{h+\omega}(P_k,T)\not=0}$. If $k\not=i$ and the maximal nonzero degree of $t_iT$ is less than $h$, then ${\Hom^{h+\omega}(P_k,t_iT)\not=0}$ by the previous assertions of this lemma, which gives a contradiction. If $k=i$, then ${\Hom^{h+\omega}(P_i,T)\not=0}$ and the maximal nonzero degree of $t_iT$ is at least $h+\omega-1>h$ by the first assertion of the current lemma. 
\end{enumerate}
\end{proof} 

Now we are ready to deduce Theorem \ref{main_thm} from Lemma \ref{princ}.

\begin{proof}[Proof of Theorem \ref{main_thm}] According to \cite[Proposition 2.3]{BT}, a group homomorphism $B_\Gamma \xrightarrow{} G$ is injective if and only if the induced monoid homomorphism $B^+_\Gamma \hookrightarrow B \xrightarrow{} G$ is injective. Hence, in our case it is sufficient to show that $B^+_\Gamma \xrightarrow{} \Aut(\mathfrak{D})$ is injective. Assume that it is not. Choose two words $\alpha$, $\beta$ with the smallest sum of lengths $l(\alpha) + l(\beta)$ among all pairs of words with coinciding images in $\Aut(\mathfrak{D})$ and $\alpha \neq \beta$. In particular, $t_\alpha(\Lambda) = T_\alpha \cong T_\beta = t_\beta(\Lambda)$ in $\mathfrak{D}$. Thus, for any $u\in\{0,1\}$, $I_{u,\alpha}=I_{u,\beta}$ and $P_i$ ($i\in V^u$) is a direct summand of $(T_\alpha)_{I_{u,\alpha}}$ if and only if it is a direct summand of $(T_\beta)_{I_{u,\beta}}$. First assume that one of $\alpha$ and $\beta$ is $1$, say $\alpha$.
Then $min(\alpha)=0$ if $\omega\le 0$ and $max(\alpha)=0$ if $\omega\ge 2$. But since $\beta \neq 1$, we have $l(\beta) \geq 1$. Lemmas \ref{tech0} and \ref{tech2} imply that $min(\beta)<0$ if $\omega\le 0$ and $max({\beta})>0$ if $\omega\ge 2$. Thus, we may assume that $\alpha \neq 1$, $\beta \neq 1$.

 \bigskip
  By Lemma \ref{princ} there exists $i \in \Gamma_0$ such that $$\alpha = s_i \alpha'\mbox{ and }\beta = s_i \beta'.$$

The images of $\alpha'$ and $\beta'$ also coincide in $\Aut(\mathfrak{D})$, and since $l(\alpha') + l(\beta') = l(\alpha) + l(\beta) - 2 < l(\alpha) + l(\beta)$ we get $\alpha' = \beta'$. But then
$ \alpha = s_i \alpha' = s_i \beta' = \beta$, 
which contradicts the assumption that $\alpha \neq \beta$.   \end{proof}


\section{The proof of Lemma \ref{princ} for $\omega\ge 2$}

In this section we assume that $\omega\ge 2$ and prove Lemma \ref{princ} in this case. Denote $max(\alpha)$ by $h_{\alpha}$.
Observe that if $\omega\ge 2$, then $P_j$ with $j\in V^u$ is a direct summand of $(T_\alpha)_{[h_{\alpha}-\omega_{u+1}+1,h_{\alpha}]}$ if and only if $\Hom^{r+\omega}(P_j,T_{\alpha})\not=0$ for some $r\in [h_{\alpha}-\omega_{u+1}+1,h_{\alpha}]$, due to the second remark right after Definition \ref{dirsum}. In fact, we used the notion of a direct summand in the form it was first provided only to unify the two cases $\omega \leq 0$ and $\omega \geq 2$ in the statement of Lemma \ref{princ} and the deduction of Theorem \ref{main_thm} from it. 

To summarize, the following statement is equivalent to Lemma \ref{princ} in the case $\omega\ge 2$: 

\begin{lemma}
If $\Hom^{r+\omega}(P_j,T_{\alpha})\not=0$ for some $r\in [h_{\alpha}-\omega_{u+1}+1,h_{\alpha}]$, then $\alpha$ is left-divisible by $s_j$.
\end{lemma}
\begin{proof} We argue by contradiction. Take $\alpha\in B_\Gamma^+$ not satisfying the claim and of minimal length among all such words. It is clear that $l(\alpha)>0$, thus, $\alpha$ can be presented as $\alpha = s_i\beta$ for some $i\in\Gamma_0$ and $\beta \in B_\Gamma^+$ with $l(\beta)<l(\alpha)$. In particular, the statement of the lemma holds for $\beta$ and all its right factors. Without loss of generality we may assume that $i\in V^0$.
 Since the assertion of the lemma fails for $\alpha$, there exists some $j \in V^u$ ($u=0,1$) such that $\Hom^{h+\omega}(P_j,T_\alpha)\not=0$ for some $h\in [h_{\alpha}-\omega_{u+1}+1,h_{\alpha}]$, but $\alpha$ is not left-divisible by $s_j$. It is clear that $j\not=i$. We may assume without loss of generality that if $\Hom^{r+\omega}(P_k,T_\alpha)\not=0$ for some $k\in V^u$ and $r>h$, then $\alpha$ is left-divisible by $s_k$.

\begin{enumerate}
\item First observe that $\beta=s_j\gamma$ for some $\gamma\in B_\Gamma^+$ with $l(\gamma)=l(\alpha)-2$.
Indeed, we have $\Hom^{h+\omega}(P_j,T_\beta)\not=0$ and $h\in [h_{\beta}-\omega_{u+1}+1,h_{\beta}]$ by Lemma \ref{tech2}. Then $\beta$ is indeed left-divisible by $s_j$ because the assertion of the lemma holds for $\beta$.

 \item Now we show that $j\in N(i)\subset V^1$ and, in particular, ${h\in [h_{\alpha}-\omega_0+1,h_{\alpha}]}$. Indeed, this easily follows from the previous claim, because if $j\not\in N(i)$, then one has $\alpha=s_is_j\gamma=s_js_i\gamma$ which contradicts the choice of $j$.

\item Note that
$$
\Hom^{h-\omega_1+1+\omega}(P_i,T_\gamma)\cong \Hom^{h+\omega}(t_it_jP_i[\omega_1-1],t_it_jT_\gamma)
=\Hom^{h+\omega}(P_j,T_\alpha)\not=0.
$$

\item\label{nns} The next step is to establish that $\Hom^{r+\omega}(P_k,T_\gamma){=}0$ for any $k\in V^1$ and $r>h$. Suppose for a contradiction that $\Hom^{r+\omega}(P_k,T_\gamma){\not=}0$ for some $k\in V^1$ and $r>h$.
By Lemma \ref{tech2} we have $h_\gamma\le h_\beta\le h_\alpha$ and $r\ge h+1\ge h_{\alpha}-\omega_0+2$.
Then $\Hom^{r+\omega}(P_k,T_\alpha)\not=0$, again by Lemma \ref{tech2}. 
Thus, $s_k$ divides $\alpha$ on the left by the choice of $h$, so we have $\alpha=s_k\alpha'$ for some $\alpha'\in B_\Gamma^+$ with $l(\alpha')=l(\alpha)-1$.
 Lemma \ref{tech2} implies that $\Hom^{h+\omega}(P_j,T_{\alpha'})\not=0$ and $h\in [h_{\alpha'}-\omega_0+1,h_{\alpha'}]$.
Since the claim of the current lemma is valid for $\alpha'$, $s_j$ divides $\alpha'$ on the left. But $s_j$ and $s_k$ commute, so $s_j$ also divides $\alpha$ on the left which contradicts the choice of $j$.

\item Let $\theta_0=\gamma$ and $\Delta_0=\varnothing$. We repeat the following procedure inductively to construct a finite sequence of pairs $(\theta_p, \Delta_p)$ with $\theta_p \in B_{\Gamma}^+, \Delta_p \subset V^0$. If at some moment $h_{\theta_p} \leq h$, we define $\theta=\theta_p$ and $\Delta=\Delta_p$.
If $h_{\theta_p} > h$, pick some $k$ such that $\Hom^{h_{\theta_p}+\omega}(P_k,T_{\theta_p})\not=0$ and define $\theta_{p+1}=s_k^{-1}\theta_p$ and $\Delta_{p+1}=\Delta_p\cup\{k\}$.
We have $\theta_{p+1}\in B_\Gamma^+$ and $l(\theta_{p+1})=l(\theta_{p})-1$, because the assertion of the lemma holds for $\theta_p$. Now recall that $\Hom^{r+\omega}(P_l,T_{\theta_0})=0$ for any $r>h$ and $l\in V^1$ by \ref{nns}). By Lemma \ref{tech2}, we have $h\ge h_{\alpha}-\omega_0+1\ge h_{\theta_p}-\omega_0+1$ for any $p\ge 0$. Thus, $\Hom^{r+\omega}(P_l,T_{\theta_p})=0$ for any $r>h$, $l\in V^1$ and $p\ge 0$ again by Lemma \ref{tech2} and by induction on $p$. Therefore $\Delta_p\subset V^0$ for any $p\ge 0$. Note also that $k\in \Delta_p$, $\theta_p\not=\theta$ implies
\begin{multline*}
\Hom^{h_{\theta_p}+\omega}(P_k,T_{\theta_p})\cong \Hom^{h_{\theta_p}+\omega}\left(\left(\prod\limits_{l\in\Delta_p}t_l\right)P_k,\left(\prod\limits_{l\in\Delta_p}t_l\right)T_{\theta_p}\right) = \Hom^{h_{\theta_p}+2\omega-1}(P_k,T_{\theta_0})=0,
\end{multline*}
because $h_{\theta_p}+\omega-1> h+\omega-1\ge h_\alpha+\omega_1>h_\alpha\ge h_{\theta_0}$. To summarize, whenever the maximal nonzero degree of $T_{\theta_p}$ is greater than $h$, we obtain $\Delta_{p+1}$ by adding a new element to $\Delta_p$. Thus, this process is bound to terminate and as a result we get a factorization $\gamma=\left(\prod\limits_{k\in\Delta}s_k\right)\theta$ with $\Delta\subset V^0$ and $\theta\in B_\Gamma^+$ such that $l(\theta)=l(\gamma)-|\Delta|$ and $h_{\theta}\le h$.

\item\label{nnf} If $i\in\Delta$, then $s_i$ divides $\gamma$ on the left. In this case  $s_is_js_i=s_js_is_j$ divides $\alpha$ on the left, and hence so does $s_j$. This contradicts the choice of $j$.

\item If $i\not\in\Delta$, then $\Hom^{h-\omega_1+1+\omega}(P_i,T_\theta)\not=0$ by Lemma \ref{tech2}. Since the claim of the current lemma is true for $\theta$, $s_i$ divides $\theta$ on the left. Then $\gamma$ is again left-divisible by $s_i$ and we get a contradiction just as before.
\end{enumerate}
\end{proof}

\begin{remark} Observe that in the case $\omega=2$ the number $h$ appearing in the proof is automatically equal to $h_{\alpha}$. This allows to omit the steps \ref{nns}--\ref{nnf} of the proof, which makes it even shorter. 
\end{remark}

\section{An outline of the proof of Lemma \ref{princ} for $\omega\le 0$}

In this section we discuss a general plan of our proof of Lemma \ref{princ} provided that $\omega\le 0$. Until the end of the paper we will assume that this is the case, although some intermediate statements will hold for any $\omega$. From here on we denote $min(\alpha)$ by $m_\alpha$.

 Similarly to the case $\omega\ge 2$, we argue by contradiction. Let $\alpha\in B_\Gamma^+$ be a word not satisfying the assertion of Lemma \ref{princ} and of minimal length. As before, $\alpha$ can be presented as $\alpha = s_i\beta$ for some $i\in\Gamma_0$ and $\beta \in B_\Gamma^+$ with $l(\beta)<l(\alpha)$. The statement of Lemma \ref{princ} holds in particular for $\beta$ and all its right factors. Without loss of generality we may assume that $i\in V^0$. Since Lemma \ref{princ} fails for $\alpha$, there exists some $j \in V^u$ ($u=0,1$) such that $P_j$ is a direct summand of $(T_\alpha)_{[m_{\alpha}, m_{\alpha} - \omega_{u+1}]}$ while $\alpha$ is not divisible by $s_j$ on the left. It is clear that $j\not=i$.
 Let $m\in [m_{\alpha}, m_{\alpha} - \omega_{u+1}]$ be the minimal degree such that $P_j$ is a direct summand of $(T_\alpha)_{m}$. We may assume without loss of generality that if $r<m$ and $P_k$ with $k\in V^u$ is a direct summand of $(T_\alpha)_{r}$, then $\alpha$ is divisible by $s_k$ on the left.  
 Now let 
 $$\Delta_{-1}=\varnothing,\,\,\Delta_0 = \{i\},\,\,\Delta_1 = \{k \in V^1 \colon P_k\mbox{ is a direct summand of }(T_\beta)_{m}\}.$$ 
 
 What can we immediately say about the set $\Delta_1$? 

\begin{lemma}\label{base} \begin{enumerate}
\item $j\in N(i)\subset V^1$.
\item If $P_k$ with $k\in V^1$ is a direct summand of $(T_{\beta})_{r}$ with $r\le m$, then $r=m$ and $k\in N(i)$. 
\item If $P_k$ is a direct summand of $(T_{\beta})_{m}$ for some  $k\in V^1\setminus\{j\}$ (i.e. $|\Delta_1| \geq 2$), then $s_l$ does not divide $\alpha$ on the left for any $l\not=i$.
\item $\Hom^r(P_i, T_\beta) = 0$  for $r\le m+\omega_0.$
\end{enumerate}
\end{lemma}
\begin{proof}
\begin{enumerate}
\item Suppose that $j \notin N(i)$. Then $P_j$ is a direct summand of $(T_\beta)_{m}$ and $m\in[m_\beta,m_\beta-\omega_{u+1}]$ by Lemma \ref{tech0}. Since the statement of Lemma \ref{princ} holds for $\beta$, we have
    $\alpha = s_i \beta = s_i s_j \beta' = s_j s_i \beta'$ for some $\beta'\in B_\Gamma^+$ with $l(\beta')=l(\alpha)-2$.
Thus, $j$ does not produce a contradiction to the statement of Lemma \ref{princ}, which contradicts the choice of $j$. Therefore, indeed, $j\in N(i)\subset V^1$ and, in particular, we have $m\in [m_{\alpha}, m_{\alpha} - \omega_0]$.

\item Let $k\in V^1$ be such that $P_k$ is a direct summand of $(T_{\beta})_{r}$ with $r\le m$. Then $P_k$ is a direct summand of $(T_{\alpha})_{r}$ by Lemma \ref{tech0}.
If $r<m$, then $s_k$ divides $\alpha$ on the left by the definition of $m$. If $r=m$ but $k\not\in N(i)$, then $s_k$ divides $\alpha$ on the left by the argument from the proof of 1). In any case, we have $\alpha=s_k\alpha'$ for some $\alpha'\in B_\Gamma^+$ with $l(\alpha')=l(\alpha)-1$.
Note that  Lemma \ref{tech0} implies again that $P_j$ is a direct summand of $(T_{\alpha'})_{m}$ and $m\in[m_{\alpha'},m_{\alpha'}-\omega_0]$.
Since the claim of Lemma \ref{princ} is valid for $\alpha'$, $s_j$ divides $\alpha'$ on the left. But $s_j$ and $s_k$ commute, so $s_j$ divides $\alpha$ on the left. This contradicts the choice of $j$.

\item Suppose that $s_l$ divides $\alpha$ on the left for some $l\not=i$. Then $s_l$ commutes with at least one of the elements $s_j$ and $s_k$. The argument from the proof of 2) shows that either $s_j$ or $s_k$ divides $\alpha$ on the left. Since $s_k$ and $s_j$ commute, $s_j$ divides $\alpha$ on the left in any case by the same argument. This contradicts the choice of $j$.

\item Suppose that $\Hom^r(P_i, T_\beta) \not= 0$  for some $r\le m+\omega_0.$
By Lemma \ref{tech0} one has $m_\alpha\le r-1\le m+\omega_0-1$. But in this case $m\not\in [m_{\alpha}, m_{\alpha} - \omega_0]$, which yields a contradiction.
\end{enumerate}
\end{proof}

In particular, by Lemma \ref{base} we know that if $|\Delta_1|\ge 2$, then for any  $k \neq i$ the word $\alpha$ cannot be written as $\alpha=s_k\alpha'$ with $l(\alpha')=l(\alpha)-1$. Hence, it suffices to show that either $|\Delta_1|=1$ and $\alpha$ is left-divisible by $s_j$, where $\{j\} = \Delta_1$, or that $|\Delta_1|\ge 2$ and $\alpha$ is left-divisible by $s_k$ with some $k \neq i$. Indeed, any one of these alternatives would immediately give a contradiction. 

Now define $s_A=\prod\limits_{k\in A}s_k$ and $t_A = \prod\limits_{k\in A}t_k$ for any $A\subseteq V^u$, $u\in\{0,1\}$. Observe that $\alpha = s_i\beta = s_{\Delta_0}\beta = s_{\Delta_0}s_{\Delta_1}\beta_1$, where $l(\beta_1) = l(\alpha) - 1 - |\Delta_1|$, because the statement of Lemma \ref{princ} holds for $\beta$. Starting from this presentation, we will continue factorizing the word $\alpha$ inductively, constructing $\Delta_2, \Delta_3,$..., as will be described below in a moment. When we process terminates, we obtain a suitable, very controllable presentation of a left factor of $\alpha$. If $|\Delta_1| = 1$, we then need to show that this left factor is left-divisible by $s_j$ (in fact, it takes just one braid relation to see this). If $|\Delta_1| \geq 2$, it suffices to show that the left factor of $\alpha$ we found by factorizing is left-divisible by some $s_k$ with $k \neq i$ (this case will require more sophisticated combinatorial arguments which are discussed in Section \ref{br}). 

We will now make this plan more precise.

\begin{enumerate}
    \item[Step I:] \underline{{\bf Factorization.}} 
    Define $\sigma_u$ for $u\in\Z$ by $\sigma_0=m$ and $\sigma_{u+1}=\sigma_u+1-\omega_{u+1}$. In other words, we have $\sigma_{2u}=m+u(2-\omega)$, $\sigma_{2u+1}=m+u(2-\omega)+1-\omega_{u+1}$. 
    
    As announced above, first we are going to construct a presentation of $\alpha$ of a particular form, starting from the 'base case' $\alpha = s_{\Delta_0}s_{\Delta_1}\beta_1 = s_is_{\Delta_1}\beta_1$ obtained earlier. Define $\chi_0(i)=1$. We will continue the process inductively and obtain a presentation $\alpha = s_{\Delta_0}s_{\Delta_1}\dots s_{\Delta_q} s_y\widetilde{\beta}$ satisfying the following conditions for $1 \leq u \leq q:$
\begin{enumerate}[label=\textnormal{(\arabic*)}]
     \item $\alpha = s_{\Delta_0}\dots s_{\Delta_u} \beta_u$ for some $\beta_u\in B_\Gamma^+$ with $l(\beta_u)=l(\alpha)-\sum\limits_{v=0}^{u}|\Delta_v|$  
  \item $\Delta_{u-2}\subseteq\Delta_u\subseteq N(\Delta_{u-1})$.
 \item $P_l$ is not a direct summand of $(T_{\beta_u})_{[\sigma_{u-3}+1,\sigma_{u-1}]}$ for any $l \in V^u$.
\item The minimal nonzero degree of $T_{\beta_u}$ is at least $\sigma_{u-2}+1$.
\item For any $l\in \Delta_u$, $P_l$ is a direct summand of $(t_lT_{\beta_u})_{[\sigma_{u-3}+1,\sigma_{u-1}]}$.
\item $\chi_u(k)=\sum\limits_{t\in N(k)}\chi_{u-1}(t)-\chi_{u-2}(k) > 0$ for any $k\in \Delta_u$, where we set $\chi_v(t)=0$ if $v<0$ or $t\not\in\Delta_v$.
\end{enumerate} 
Moreover, $y\in \Delta_{q-1}$, $s_y$ divides $\beta_q$ on the left and $$\chi_{q+1}(y) = \sum\limits_{t\in N(y)\cap \Delta_{q}}\chi_{q}(t)-\chi_{q-1}(y) = 0.$$
  
 Note that the sets $\Delta_0$ and $\Delta_1$ defined earlier satisfy the required conditions. Indeed, one has $\Delta_1 \subseteq N(i) = N(\Delta_0)$ by Lemma \ref{base}. The minimal nonzero degree of $T_{\beta_1}$ is not smaller than $m_\alpha\ge m+\omega_0=\sigma_{-1}+1$ by  Lemma \ref{tech0}. Since $P_j$ is a direct summand of $(T_\beta)_{m}$, Lemma \ref{tech0} implies that $P_j$ is also a direct summand of $\left(t_{\Delta_1\setminus\{j\}}^{-1}T_\beta\right)_{m}=(t_{j}T_{\beta_1})_{m}$ for any $j\in\Delta_1$. 
Moreover, for any $r\le m=\sigma_0$, $P_k$ with $k\in V^1\setminus\Delta_1$ is not a direct summand of $(T_{\beta})_{r}$ by Lemma \ref{base}, hence is not a direct summand of $(T_{\beta_1})_{r}$ by Lemma \ref{tech0}. On the other hand, $P_k$ with $k\in\Delta_1$ cannot be a direct summand of $(T_{\beta_1})_{r}$ either, because we have $r+\omega-1\le m+\omega-1\le m_{\beta}+\omega_1-1<m_{\beta}$ and, therefore,
$$\Hom^{r+\omega}(P_k,T_{\beta_1})\cong \Hom^{r+2\omega-1}(t_{\Delta_1}P_k[\omega-1],T_{\beta})=\Hom^{r+2\omega-1}(P_k,T_{\beta})=0.$$
 Finally, we clearly have $\chi_1(j)=1>0$ for any $j\in\Delta_1$.

  Thus, it is sufficient to show that if we have sets $\Delta_0,\dots,\Delta_{p}$ such that the properties above are satisfied for any $1\le u\le p$, then we either can construct $\Delta_{p+1}$ in such a way that the properties above will be satisfied for $u=p+1$ as well or find $y\in \Delta_{p-1}$ such that $s_y$ divides $\beta_p$ on the left and $\sum\limits_{t\in N(y)\cap \Delta_{p}}\chi_{p}(t)=\chi_{p-1}(y)$.

    We will introduce all the necessary technical tools in Section \ref{tto} and then discuss this step in detail in Section \ref{2tpt}. 
    
    \item[Step II:] \underline{{\bf Braiding.}} As we have explained a little earlier, once a presentation of $\alpha$ of the form $s_{\Delta_0}s_{\Delta_1}\dots s_{\Delta_p} s_y\widetilde{\beta}$ is obtained, it remains to show that either $\Delta_1=\{j\}$ for some $j$ such that $s_j$ divides $s_{\Delta_0}s_{\Delta_1}\dots s_{\Delta_p} s_y$ on the left or $|\Delta_1|\ge 2$ and at least one of $s_k$ with $k \in \Gamma_0 \setminus \{i\}$ divides $_{\Delta_0}s_{\Delta_1}\dots s_{\Delta_p} s_y$ on the left. Now note that if $\Delta_1=\{j\}$, then $\chi_2(i)=0$. Thus, $y = i$ and our presentation is of the form $\alpha=s_{i}s_{j}s_i\widetilde{\beta}=s_{j}s_{i}s_j\widetilde{\beta}$, which immediately yields a contradiction. Thus, at this point it suffices to consider the case $|\Delta_1|\ge 2$ and show that some $k \in \Gamma_0 \setminus \{i\}$ can be pulled to the very left of the subword $s_{\Delta_0}s_{\Delta_1}\dots s_{\Delta_p} s_y$, applying a sequence of braid and commutation relations. This step is discussed in Section \ref{br}.  
\end{enumerate}

 \section{Two-term objects}\label{tto}

In this section we introduce the notion of a two-term object in $\mathfrak{D}$ and discuss some of the properties of these objects. We will extensively use the results of this sections in the 'factorization' step of the proof. To put it simply, a two-term object is a cone of a morphism from a direct sum of $P_i$'s with $i \in V^u$ to a direct sum of $P_i$'s with $i \in V^{u+1}$, for $u \in \{0, 1\}$, modulo an appropriate shift. When $\mathfrak{D}$ is the bounded derived category $D^b(\Lambda)$ of a finite-dimensional algebra (see Section \ref{trpic}), indecomposable two-term objects are direct summands of two-term tilting complexes over $\Lambda$. In our proof of the main result two-term objects will appear naturally as images of spherical objects under some sequences of spherical twists. For instance, as we will see later, in the notation of the previous section the object $t_{\Delta_u}^{-1} \dots t_{\Delta_1}^{-1} P_i$ is a shift of a two-term object. 

\begin{definition}
An object $X$ of a triangulated category $\mathfrak{D}$ with a fixed $\Gamma$-configuration of $\omega$-spherical objects $\{P_j\}_{j\in \Gamma_0}$ is called {\it two-term} if there exists a triangle
$$X[-1]\xrightarrow{\beta_X} \bigoplus\limits_{j \in V^u } P_j^{x_j}[-\omega_u] \xrightarrow{\varphi_X} \bigoplus\limits_{k \in V^{u+1}} P_k^{x_k}\xrightarrow{\alpha_X} X $$
in $\mathfrak{D}$ for some $u\in\{0,1\}$ and some non-negative integers $x_j$, $j\in\Gamma_0$.
\end{definition} 

We will be especially interested in two-term objects for which none of the $P_i$'s either ``on the left'' or ``on the right'' splits off as a direct summand: 

\begin{definition} \begin{enumerate}
    \item A two-term object $X$ is called {\it right-proper} if $f\varphi_X\not=0$ for any split epimorphism ${f \colon \bigoplus\limits_{k \in V^{u+1}} P_k^{x_k}\rightarrow P_l}$ with $l\in V^{u+1}$.
    \item 
A two-term object $X$  is called {\it left-proper} if $\varphi_X g[-\omega_u]\not=0$ for any split monomorphism ${g \colon P_l\rightarrow \bigoplus\limits_{j \in V^u } P_j^{x_j}}$ with $l\in V^{u}$. 
\end{enumerate}

For a two-term object $X$ we also define $\ls(X)=\{j\in V^{u}\mid x_j\not=0\}$ and $\rs(X)=\{k\in V^{u+1}\mid x_k\not=0\}$.
\end{definition}

For instance, $P_l$ for any $l\in\Gamma_0$ is the simplest example a left-proper two-term object. Here we have $\alpha_{P_l}=id_{P_l}$, $\beta_{P_l}=\varphi_{P_l}=0$, $\ls(P_l)=\varnothing$ and $\rs(P_l)=\{l\}$. It is not difficult to show that any two-term object $X$ can be represented in the form $X=X'\oplus \bigoplus\limits_{k \in V^{u+1}} P_k^{r_k}$, where $X'$ is a right-proper two-term object. 

The following lemma shines a little more light on these two properties, providing a couple of useful equivalent reformulations. 


\begin{lemma}\label{lpeq} Let $X$ be a two-term object as in the definition above.
\begin{enumerate}
\item The following three conditions are equivalent:
\begin{itemize}
\item $X$ is right-proper;
\item if $l\in V^{u+1}$ and $g \colon P_l\rightarrow \bigoplus\limits_{k \in V^{u+1}} P_k^{x_k}$ is not a split monomorphism, then ${\alpha_Xg=0}$;
\item $\dim_\kk\Hom^*(X,P_l)=\sum\limits_{k\in N(l)}x_k$ for any $l\in V^{u+1}$.
\end{itemize}
\item The following three conditions are equivalent:
\begin{itemize}
\item $X$ is left-proper;
\item if $l\in V^{u}$ and $f:\bigoplus\limits_{j \in V^u } P_j^{x_j}\rightarrow P_l$ is not a split epimorphism, then ${f[-\omega_u]\beta_X=0}$;
\item $\dim_\kk\Hom^*(X,P_l)=\sum\limits_{k\in N(l)}x_k$ for any $l\in V^{u}$.
\end{itemize}
\end{enumerate}
\end{lemma}
\begin{proof}  We will prove only the first claim, the second one can be deduced by dual arguments.
 Fix some $l\in V^{u+1}$. Then we have 
\begin{multline*}
\dim_\kk\Hom^*(X,P_l)=\dim_\kk\Hom^*\left(\bigoplus\limits_{j \in V^u } P_j^{x_j},P_l\right)+\dim_\kk \Hom^*\left(\bigoplus\limits_{k \in V^{u+1}} P_k^{x_k},P_l\right)-2\dim_\kk\im\varphi_X^*\\
=\sum\limits_{k\in N(l)}x_k+2\left(x_l-\dim_\kk\im\varphi_X^*\right),
\end{multline*}
where $\varphi_X^* \colon \Hom^*\left(\bigoplus\limits_{k \in V^{u+1}} P_k^{x_k},P_l\right)\rightarrow \Hom^*\left(\bigoplus\limits_{j \in V^u } P_j^{x_j},P_l\right)$ is the map induced by $\varphi_X$.
Note that the set $\Hom^*\left(\bigoplus\limits_{k \in V^{u+1}} P_k^{x_k},P_l\right)=\Hom^*(P_l^{x_l},P_l)$ is a direct sum of two spaces of dimension $x_l$, one of which is annihilated by $\varphi_X$ and the other is formed by split epimorphisms.
Thus, $\dim_\kk\im\varphi_X^*\le x_l$ and the equality holds for all $l\in V^{u+1}$ precisely when $X$ is right-proper. Thus, $X$ is right-proper if and only if ${\dim_\kk\Hom^*(X,P_l)=\sum\limits_{k\in N(l)}x_k}$ for any $l\in V^{u+1}$.
Observe that $\dim_\kk\Hom^*(X,P_l)=\dim_\kk\Hom^*(P_l,X)$ because $P_l$ is spherical. Analogous arguments show that  $\dim_\kk\Hom^*(P_l,X)=\sum\limits_{k\in N(l)}x_k$ for any $l\in V^{u+1}$ if and only if $\alpha_Xg=0$ for any
$g \colon P_l\rightarrow \bigoplus\limits_{k \in V^{u+1}} P_k^{x_k}$ that is not a split monomorphism.
\end{proof}

Next we are going to show that the set of two-term object is stable under some autoequivalences of $\mathfrak{D}$ (namely, some compositions of spherical twists and their inverses). For $\Delta \subseteq V^u$, we denote $t_{\Delta}^+=t_{\Delta}[\omega_{u+1}-1]$ and ${t_{\Delta}^-=t_{\Delta}^{-1}[1-\omega_{u+1}]}$.

\begin{lemma}\label{2toe} Let $X$ be a two-term object.
\begin{enumerate}
\item If $X$ is right-proper and $\rs(X)\subseteq\Delta \subseteq V^{u+1}$, then $t_{\Delta}^+X$ is a left-proper two-term object with the defining triangle of the form
$$t_{\Delta}^+X[-1]\xrightarrow{\beta_{t_{\Delta}^+X}} \bigoplus\limits_{k \in V^{u+1}} P_k^{x_k'}[-\omega_{u+1}]\xrightarrow{\varphi_{t_{\Delta}^+X}}\bigoplus\limits_{j \in V^u } P_j^{x_j}  \xrightarrow{\alpha_{t_{\Delta}^+X}} t_{\Delta}^+X,$$
 where $x_k' = \sum\limits_{j\in N(k)} x_j - x_k$ for $k \in \Delta$ and $x_k'=0$ for $k\in V^{u+1}\setminus\Delta$.

\item  If $X$ is left-proper and $\ls(X)\subseteq\Delta \subseteq V^{u}$, then $t_{\Delta}^-X$ is a right-proper two-term object with the defining triangle of the form
$$t_{\Delta}^-X[-1]\xrightarrow{\beta_{t_{\Delta}^-X}} \bigoplus\limits_{k \in V^{u+1}} P_k^{x_k}[-\omega_{u+1}]\xrightarrow{\varphi_{t_{\Delta}^-X}}\bigoplus\limits_{j \in V^u } P_j^{x_j'}  \xrightarrow{\alpha_{t_{\Delta}^-X}} t_{\Delta}^-X,$$
 where $x_j' = \sum\limits_{k\in N(j)} x_k - x_j$ for $j \in \Delta$ and $x_j'=0$ for $j\in V^u\setminus\Delta$.
\end{enumerate}
\end{lemma}
\begin{proof}
As for the previous Lemma, we will prove only the first assertion, since  second one can be deduced by dual arguments. We fix some $l\in V^{u+1}$. By Lemma \ref{lpeq}, $\Hom^*(P_l,X)$ is of dimension $\sum\limits_{k\in N(l)}x_k$. Hence, it is generated by (1) $x_l$ compositions $P_l\rightarrow \bigoplus\limits_{k \in V^{u+1}} P_k^{x_k}\xrightarrow{\alpha_X}X$, where the first arrow ranges over $x_l$ linearly independent direct inclusions, and (2) $x_l'$ maps $P_l\rightarrow X[\omega-1]$ that after composition with $\beta_X[\omega]$ give $x_l'$ linearly independent maps from $P_l$ to $\bigoplus\limits_{j \in V^u } P_j^{x_j}[\omega_{u+1}]$ annihilated by $\varphi_X[\omega]$. Taking all these morphisms for all $l$ together we get a map 
$$\bigoplus\limits_{k \in V^{u+1}} P_k^{x_k}[\omega_u-1]\oplus \bigoplus\limits_{k \in V^{u+1}} P_k^{x_k'}[-\omega_{u+1}]\xrightarrow{\begin{pmatrix}\alpha_X { [\omega_u-1]}&\gamma\end{pmatrix}}X[\omega_u-1]$$
whose cone is isomorphic to $t_{\Delta}^+X$. Now, applying the octahedral axiom to the composition $\begin{pmatrix}\alpha_X { [\omega_u-1]} &\gamma\end{pmatrix}\circ\begin{pmatrix}id_{Z}\\0\end{pmatrix}=\alpha_X { [\omega_u-1]}$, where $Z$ denotes $\bigoplus\limits_{k \in V^{u+1}} P_k^{x_k}[\omega_u-1]$, we get the following commutative diagram whose rows and columns are triangles

\begin{center}

 {\scriptsize
\begin{tikzcd}
{\bigoplus\limits_{k \in V^{u+1}} P_k^{x_k}[\omega_u-1]} \arrow[r, "\begin{pmatrix}id_{Z}\\0\end{pmatrix}"] \arrow[equal, d] &
\bigoplus\limits_{k \in V^{u+1}} P_k^{x_k}[\omega_u-1]\oplus \bigoplus\limits_{k \in V^{u+1}} P_k^{x_k'}[-\omega_{u+1}]\arrow[r] \arrow[d, "(\alpha_X { {[\omega_u-1]}} \enskip \gamma)"] & \bigoplus\limits_{k \in V^{u+1}} P_k^{x_k'}[-\omega_{u+1}] \arrow[d, "\varphi_{t_\Delta^+X}"] \\
\bigoplus\limits_{k \in V^{u+1}} P_k^{x_k}[\omega_u-1] \arrow[r, "\alpha_X { [\omega_u-1]}"] & X[\omega_u-1] \arrow[r, "\beta_X{[-\omega_u]}"]\arrow[d, "\alpha_{t_\Delta^+X}\beta_X {[-\omega_u]}"]& \bigoplus\limits_{j \in V^u } P_j^{x_j}\arrow[d, "\alpha_{t_\Delta^+X}"] \\ 
& t_{\Delta}^+X\arrow[equal, r] & t_{\Delta}^+X 
\end{tikzcd}
}

\end{center}

 Thus, $t_{\Delta}^+X$ indeed has the required form. A direct inclusion $g \colon P_l\rightarrow \bigoplus\limits_{k \in V^{u+1}} P_k^{x_k'}$ such that the composition
$$
P_l[-\omega_{u+1}]\xrightarrow{g[-\omega_{u+1}]} \bigoplus\limits_{k \in V^{u+1}} P_k^{x_k'}[-\omega_{u+1}]\xrightarrow{\varphi_{t_\Delta^+X}} \bigoplus\limits_{j \in V^u } P_j^{x_j}
$$
is zero would produce a linear dependence between $x_l'$ components of the morphism ${P_l^{x_l'}[-\omega_{u+1}]\rightarrow \bigoplus\limits_{j \in V^u } P_j^{x_j}}$. But they are linearly independent by our construction. Thus, $t_{\Delta}^+X$ is indeed left-proper.
\end{proof}

Now we introduce a very natural definition of {\it a two-term subobject} of a two-term object. We will then study the behaviour of this and some other relations between two-term objects with respect to autoequivalences of the form $t_{\Delta}^{\pm}$. 

\begin{definition} \begin{enumerate}
    \item Let 
$${\small X=cone\left(\bigoplus\limits_{j \in V^u } P_j^{x_j}[-\omega_u] \xrightarrow{\varphi_X} \bigoplus\limits_{k \in V^{u+1}} P_k^{x_k}\right)}\mbox{ and }{\small Y=cone\left(\bigoplus\limits_{j \in V^u } P_j^{y_j}[-\omega_u] \xrightarrow{\varphi_Y} \bigoplus\limits_{k \in V^{u+1}} P_k^{y_k}\right)}$$ be two-term objects. We say that $X$ is a {\it two-term subobject} of $Y$ if there exist split monomorphisms $\iota_u \colon \bigoplus\limits_{j \in V^u } P_j^{x_j}\hookrightarrow \bigoplus\limits_{j \in V^u } P_j^{y_j}$ and $\iota_{u+1} \colon \bigoplus\limits_{k \in V^{u+1} } P_k^{x_k}\hookrightarrow \bigoplus\limits_{k \in V^{u+1} } P_k^{y_k}$ such that $\iota_{u+1}\varphi_X=\varphi_Y\iota_u[-\omega_u]$.

The two-term subobject $X$ of $Y$ is called {\it trivial} if either $X=0$ or both of the maps $\iota_u$ and $\iota_{u+1}$ are isomorphisms. Otherwise $X$ is called a {\it nontrivial two-term subobject} of $Y$. 
\item 
We say that ${f\colon X\rightarrow Y[\omega]}$ is a {\it right socle morphism} if it of the form ${f=\alpha_Y[\omega]f'}$ for some ${f':X\rightarrow \bigoplus\limits_{k \in V^{u+1}} P_k^{y_k}[\omega]}$ such that for any split epimorphism $g:\bigoplus\limits_{k \in V^{u+1}} P_k^{y_k}\rightarrow P_l$ with $l\in V^{u+1}$ the morphism ${g[\omega]f'\alpha_X\colon\bigoplus\limits_{k \in V^{u+1}} P_k^{x_k}\rightarrow P_l[\omega]}$ is not a split epimorphism anymore. 
\end{enumerate}

\end{definition}

\begin{remark} The second condition in the definition of a right socle morphism is valid automatically if $X$ is right-proper or $\omega\not=0$. Moreover, if $\omega\not=0$ and $X$ is left-proper, then any morphism of the form $X\xrightarrow{f}Y[\omega]$ is automatically right socle, because in this case $\Hom_\mathfrak{D}(P_k,P_j[1+\omega_{u+1}])=0$ for any $k\in V^{u+1}$, $j\in V^u$ and $g[1-\omega_u]\beta_X[1]=0$ for any $g\colon\bigoplus\limits_{j \in V^u } P_j^{x_j}\rightarrow \bigoplus\limits_{j \in V^u } P_j^{y_j}[\omega]$. In fact, the definition of a right socle morphism is introduced solely to cover the case $\omega=0$ which nevertheless is of special interest to us in view of an application to the derived Picard groups of algebras. As it is easy to see, most of the claims about right socle morphisms we provide below are trivial for $\omega \neq 0$.
\end{remark}

\begin{lemma}\label{presinc} Let $X$ and $Y$ be as above. Suppose also that both $X$ and $Y$ are right-proper. If $X$ is a nontrivial two-term subobject of $Y$, then $t_{\Delta}^+X$ is a non-trivial two-term subobject of $t_{\Delta}^+Y$ for any $\rs(Y)\subseteq\Delta \subseteq V^{u+1}$.
\end{lemma}
\begin{proof}
It is clear that $\rs(X)\subseteq\rs(Y)$, and hence both objects $t_{\Delta}^+X$ and $t_{\Delta}^+Y$ are two-term by Lemma \ref{2toe}. Moreover, recall from the proof of Lemma \ref{2toe} that one has 
$$t_{\Delta}^+X=cone\left(\bigoplus\limits_{k \in V^{u+1}} P_k^{x_k'}[-\omega_{u+1}]\xrightarrow{\varphi_{t_{\Delta}^+X}}\bigoplus\limits_{j \in V^u } P_j^{x_j}\right)$$
where $x_l'=0$ for any $l\in V^{u+1}\setminus\Delta$ and for any $l\in \Delta$ the components of the map $\varphi_{t_{\Delta}^+X}|_{P_l^{x_l'}}$ constitute the basis of the vector space $\Ker\Hom_\mathfrak{D}(P_l,\varphi_X)$.
 The morphism $\varphi_{t_{\Delta}^+Y}$ satisfies analogous conditions.

\begin{center}
\begin{tikzcd}
P_l \arrow[r, hook, "g"] & \bigoplus\limits_{k \in V^{u+1}} P_k^{x_k'}[-\omega_{u+1}] \arrow[r, "\varphi_{t_\Delta^+X}"] \arrow[d, "\iota_{u+1}'"] & \bigoplus\limits_{j \in V^{u}} P_j^{x_j} \arrow[r, "\varphi_X{[\omega_u]}"] \arrow[d, hook, "\iota_u"]  &
\bigoplus\limits_{k \in V^{u+1}} P_k^{x_k}[\omega_u] \arrow[d, hook, "\iota_{u+1}{[\omega_u]}"]  \\ & \bigoplus\limits_{k \in V^{u+1}} P_k^{y_k'}[-\omega_{u+1}] \arrow[r, "\varphi_{t_\Delta^+Y}"] 
& \bigoplus\limits_{j \in V^{u}} P_j^{y_j} \arrow[r,"\varphi_Y{[\omega_u]}"] & \bigoplus\limits_{k \in V^{u+1}} P_k^{y_k}[\omega_u] \\
\end{tikzcd} 
\end{center}

 Since $\varphi_Y{[\omega_u]}\iota_u\varphi_{t_{\Delta}^+X}=0$, there exists a map $\iota_{u+1}'\colon \bigoplus\limits_{k \in V^{u+1}} P_k^{x_k'}\rightarrow \bigoplus\limits_{k \in V^{u+1}} P_k^{y_k'}$ such that $\iota_u\varphi_{t_{\Delta}^+X}=\varphi_{t_{\Delta}^+Y}\iota_{u+1}'[-\omega_{u+1}]$.
It remains to show that $\iota_{u+1}'$ is a split monomorphism. Suppose it is not. Then for some $l\in V^{u+1}$ there exists a direct inclusion $g:P_l\rightarrow \bigoplus\limits_{k \in V^{u+1}} P_k^{x_k'}$ such that $\iota_{u+1}'g$ is not a direct inclusion. In this case we have $\iota_u\varphi_{t_{\Delta}^+X}g=\varphi_{t_{\Delta}^+Y}\iota_{u+1}'g=0$, and hence $\varphi_{t_{\Delta}^+X}g=0$. But by Lemma \ref{2toe} we know that $t_{\Delta}^+X$ is left-proper. This contradiction implies that $\iota_{u+1}'$ is a split monomorphism, and hence we are done.
\end{proof}

\begin{lemma}\label{rightred} Let $X$ and $Y$ be as in the previous Lemma. Then any right socle morphism $f:X\rightarrow Y[\omega]$ factors through some right socle morphism $X'\rightarrow Y[\omega]$, where $X'$ is a two-term object such that $\rs(X')\subset\rs(Y)$.
\end{lemma}
\begin{proof} Let $f=\alpha_Y[\omega]f':X\rightarrow Y{[\omega]}$ be a decomposition of the right socle morphism $f$ and let ${\Delta=\rs(X)\setminus\rs(Y)}$. By $\iota$ we denote the direct inclusion ${\bigoplus\limits_{k \in \Delta} P_k^{x_k}\hookrightarrow\bigoplus\limits_{k \in V^{u+1}} P_k^{x_k}}$ and by $\pi$ the split epimorphism ${\bigoplus\limits_{k \in V^{u+1}} P_k^{x_k}\twoheadrightarrow \bigoplus\limits_{k \in V^{u+1}\setminus\Delta} P_k^{x_k}}$. Since ${\Hom_\mathfrak{D}\left(\bigoplus\limits_{k \in \Delta} P_k^{x_k}, \bigoplus\limits_{k \in V^{u+1}} P_k^{y_k}{[\omega]}\right)=0}$, one has $f'\alpha_X\iota=0$, and hence $f'$ factors through some morphism $f'':X'\rightarrow \bigoplus\limits_{k \in V^{u+1}} P_k^{y_k}[\omega]$, where $X'=cone(\alpha_X\iota)$. Now by the octahedral axiom we have  
$$X'\cong cone\left(\bigoplus\limits_{j \in V^u } P_j^{x_j}[-\omega_u] \xrightarrow{\pi\varphi_X} \bigoplus\limits_{k \in V^{u+1}\setminus\Delta} P_k^{x_k}\right),$$
i.e. $X'$ is as required.  

\begin{center}
\begin{tikzcd}
& \bigoplus\limits_{j \in V^u} P_j^{x_j} [-\omega_u] \arrow[r, bend right=1, equal] \arrow[d, "\varphi_X"] &  \bigoplus\limits_{j \in V^u} P_j^{x_j} [-\omega_u]  \arrow[d, "\pi \varphi_X"] \\ 
\bigoplus\limits_{k\in\Delta} P_k^{x_k} \arrow[r, hook, "\iota"] \arrow[d, equal] & \bigoplus\limits_{k \in V^{u+1}} P_k^{x_k} \arrow[r, two heads, "\pi"] \arrow[d, "\alpha_X"] & \bigoplus\limits_{k \in V^{u+1}\setminus \Delta} P_k^{x_k} \arrow[d, "\alpha_{X'}"] \\ 
\bigoplus\limits_{k\in\Delta} P_k^{x_k} \arrow[r, "\alpha_X \iota"] & X \arrow[r] \arrow[dr, "f'"] & X' \arrow[d, "f''"] \\ 
& & \bigoplus\limits_{k\in V^{u+1}} P_k^{y_k} [\omega] \arrow[r, "g{[\omega]}"]& P_l[\omega]
\end{tikzcd}
\end{center}

It remains to show that the morphism $\alpha_Y[\omega]f''$ is right socle. Suppose that $g \colon { \bigoplus\limits_{k \in V^{u+1}} P_k^{y_k}}\rightarrow P_l$ is a split epimorphism. We need to show that $g{[ \omega]}f''\alpha_{X'}$ is not a split epimorphism. Observe that $\pi$ is an epimorphism and $gf'\alpha_X$ is not a split epimorphism, since $f$ is right socle. On the other hand, we have $f''\alpha_{X'}\pi=f'\alpha_X$ by construction. Hence $g{[\omega]}f'' \alpha_{X'}$ is not a split epimorphism as well.

\end{proof}

\begin{lemma}\label{simpvan} If $X$ is right-proper and $l\in V^{u+1}$, then any right socle morphism from $P_l$ to $X[\omega]$ is zero.
\end{lemma}
\begin{proof} Let $f=\alpha_X[\omega]f':P_l\rightarrow X$ be a right socle morphism. Then $f'$ is not a split monomorphism by the definition of a right socle morphism. Then $f=\alpha_X[\omega]f'=0$ by Lemma \ref{lpeq}.
\end{proof}

\begin{lemma}\label{pressoc} Let $X$ and $Y$ be as above. Suppose that both $X$ and $Y$ are right-proper and $\rs(X)\subset \rs(Y)$. Then for any right socle $f:X\rightarrow Y[\omega]$ and any $\rs(Y)\subseteq\Delta \subseteq V^{u+1}$ the morphism $t_{\Delta}^+f:t_{\Delta}^+X\rightarrow t_{\Delta}^+Y[\omega]$ is right socle as well.
\end{lemma}
\begin{proof} To simplify the notations here we consider only the case $\omega=0$. Other cases are obvious due to the remark we made after the definition of a right socle morphism. Note that ${\omega_u=\omega_{u+1}=0}$ by our assumption.
Due to the proof of Lemma \ref{2toe} there is a triangle
$$
 t_{\Delta}^+X[-1]\xrightarrow{\begin{pmatrix}\psi\\ \beta_{t_{\Delta}^+X}\end{pmatrix}}\bigoplus\limits_{k \in V^{u+1}} P_k^{x_k}[-1]\oplus \bigoplus\limits_{k \in V^{u+1}} P_k^{x_k'}\xrightarrow{\begin{pmatrix}\alpha_X&\gamma\end{pmatrix}} X[-1]\xrightarrow{\alpha_{t_{\Delta}^+X}\beta_X}t_{\Delta}^+X
$$
emerging from the definition of a spherical twist.  There is also an analogous triangle for $Y$. Observe that $\beta_Y[1]f = 0$, because $f$ factors through $\alpha_Y$. Then $t_{\Delta}^+f$ satisfies the condition
$(t_{\Delta}^+f)\alpha_{t_{\Delta}^+X}\beta_X=\alpha_{t_{\Delta}^+Y}\beta_Yf[-1]=0.$

 Thus, $t_\Delta^+f$ factors through $\begin{pmatrix}\psi[1]\\ \beta_{t_{\Delta}^+X} [1] \end{pmatrix}$, i.e. 
$\beta_{t_{\Delta}^+Y}[1]t_{\Delta}^+f=(\beta_{t_{\Delta}^+Y}[1]f')\begin{pmatrix}\psi[1]\\ \beta_{t_{\Delta}^+X}[1]\end{pmatrix}$
 for some morphism
$f' \colon \bigoplus\limits_{k \in V^{u+1}} P_k^{x_k}\oplus \bigoplus\limits_{k \in V^{u+1}} P_k^{x_k'}[1]\rightarrow t_{\Delta}^+Y.$
Since $\beta_{t_{\Delta}^+Y}[1]f'$ is annihilated by $\varphi_{t_{\Delta}^+Y}{ [1]}$ and $t_{\Delta}^+Y$ is left-proper, all components of $\beta_{t_{\Delta}^+Y}[1]f'$ are not isomorphisms, i.e. all morphisms
$\bigoplus\limits_{k \in V^{u+1}} P_k^{x_k}\oplus \bigoplus\limits_{k \in V^{u+1}} P_k^{x_k'}[1]\rightarrow P_l[1]$
constituting $\beta_{t_{\Delta}^+Y}[1]f'$ are not split epimorphisms.
Remark that ${\beta_{t_{\Delta}^+X}[1]}$ annihilates all morphisms of the form ${\bigoplus\limits_{k \in V^{u+1}} P_k^{x_k'}[1]\rightarrow P_l[1]}$ with $l\in V^{u+1}$ that are not split epimorphisms by the left-properness of $t_{\Delta}^+X$.

Since $\Hom_\mathfrak{D}\left(\bigoplus\limits_{k \in V^{u+1}} P_k^{x_k},\bigoplus\limits_{k \in V^{u+1}} P_k^{y_k}[1]\right)=0$, we have $\beta_{t_{\Delta}^+Y}[1]t_{\Delta}^+f=0$. Therefore, $t_{\Delta}^+f=\alpha_{t_{\Delta}^+Y} f''$ for some $f'' \colon t_{\Delta}^+X\rightarrow \bigoplus\limits_{j \in V^{u}} P_j^{y_j}$. Suppose now that $gf''\alpha_{t_{\Delta}^+X}$ is a split epimorphism for some $g \colon \bigoplus\limits_{j \in V^{u}} P_j^{y_j}\rightarrow P_l$. Then $gf''\alpha_{t_{\Delta}^+X}\beta_X\not=0$, and hence ${f''\alpha_{t_{\Delta}^+X}\beta_X\not=0}$. On the other hand, we have proved earlier that $\alpha_{t_{\Delta}^+Y}f''\alpha_{t_{\Delta}^+X}\beta_X=(t_{\Delta}^+f)\alpha_{t_{\Delta}^+X}\beta_X=0$.
Then $f''\alpha_{t_{\Delta}^+X}\beta_X$ factors as $\varphi_{t_\Delta^+ Y} \theta $ for some morphism ${\theta \colon X[-1] \rightarrow \bigoplus\limits_{k \in V^{u+1}} P_k^{y_k'}}$.
 Since $\Hom_\mathfrak{D}\left(\bigoplus\limits_{k \in V^{u+1}} P_k^{x_k}[-1],\bigoplus\limits_{k \in V^{u+1}} P_k^{y_k'}\right)=0$, $\theta$ 
factors through $\beta_X$. Thus,  we have $f''\alpha_{t_{\Delta}^+X}\beta_X=\varphi_{t_\Delta^+Y} \theta =  \varphi_{t_{\Delta}^+Y}f'''\beta_X$ for some $f''' \colon \bigoplus\limits_{j \in V^{u}} P_j^{x_j} \rightarrow\bigoplus\limits_{k \in V^{u+1}} P_k^{y_k'}$. All components of $\varphi_{t_{\Delta}^+Y}f'''$ are not isomorphisms, so the morphism $g(f''\alpha_{t_{\Delta}^+X}-\varphi_{t_{\Delta}^+Y}f''')$ is a split epimorphism, and hence $(f''\alpha_{t_{\Delta}^+X}-\varphi_{t_{\Delta}^+Y}f''')\beta_X$ cannot be zero. We have a contradiction, which finishes the proof of the lemma.
\end{proof}

\section{Factorization}\label{2tpt}

We return to the context of Lemma \ref{princ}. Until the end of the paper we assume that $\omega\le 0$. Recall that  $\sigma_{2u}=m+u(2-\omega)$, ${\sigma_{2u+1}=m+u(2-\omega)+1-\omega_1}$. Now suppose that we have sets $\Delta_0,\dots,\Delta_p$ for some $p\ge 1$ such that $\Delta_0=\{i\}$ for some $i\in V^0$.
Let us recall that the numbers $\chi_u(k)$ ($0\le u\le p$, $k\in\Delta_u$) are defined inductively by $\chi_0(i):=1$ and $\chi_u(k):=\sum\limits_{t\in N(k)}\chi_{u-1}(t)-\chi_{u-2}(k)$ for $u\ge 1$, where we set for convenience $\chi_v(t)=0$ if $v<0$ or $t\not\in\Delta_v$.
Suppose also that the following conditions hold for $1 \leq u \leq p:$
\begin{enumerate}[label=\textnormal{(\arabic*)}]
     \item\label{factor} $\alpha = s_{\Delta_0}\dots s_{\Delta_u} \beta_u$ for some $\beta_u\in B_\Gamma^+$ with $l(\beta_u)=l(\alpha)-\sum\limits_{v=0}^{u}|\Delta_v|$
  \item\label{lying} $\Delta_{u-2}\subseteq\Delta_u\subseteq N(\Delta_{u-1})$, where we set $\Delta_{-1}=\varnothing$ for convenience.
  \item\label{nds} $P_l$ is not a direct summand of $(T_{\beta_u})_{[\sigma_{u-3}+1,\sigma_{u-1}]}$ for any $l \in V^u$.
\item\label{md} The minimal nonzero degree of $T_{\beta_u}$ is not smaller than $\sigma_{u-2}+1$.
\item\label{ds} For any $l\in \Delta_u$, $P_l$ is a direct summand of $(t_lT_{\beta_u})_{[\sigma_{u-3}+1,\sigma_{u-1}]}$.
    \item\label{taunz} $\chi_u(k) > 0$ for any $k\in \Delta_u$.
\end{enumerate}

We want to continue the process and construct $\Delta_{p+1}$ in such a way that  the conditions \ref{factor} and \ref{lying} hold for $u=p+1$ and, whenever the condition \ref{taunz} holds for $u=p+1$, then so do the conditions \ref{nds}--\ref{ds}. Define 
$$C_u:=  t_{\Delta_u}^- \dots t_{\Delta_1}^- P_i =  (t_{\Delta_u}^{-1} \dots t_{\Delta_1}^{-1} P_i)[\sigma_u-m+\omega_u-\omega_0].$$

The first crucial fact that we will need is that $C_u$ is a two-term object of a certain form.

\begin{lemma}\label{Cis2t} For any $1\le u\le p$ there exists a triangle of the form
$$C_u[-1]\xrightarrow{\beta_u}\bigoplus_{j \in \Delta_{u-1}} P_j^{\chi_{u-1}(j)}[-\omega_{u-1}] \xrightarrow{\varphi_{u}} \bigoplus_{k \in \Delta_u} P_k^{\chi_u(k)}\xrightarrow{\alpha_u}C_u.$$ \end{lemma} 
\begin{proof} Since $P_i$ is a left-proper two-term object, the required triangle can be obtained by iterated application of Lemma \ref{2toe}.  \end{proof}

The next proposition elaborates on the properties of the objects $C_u$ that will allow us to find direct summands of $(T_{\beta_u})_{[\sigma_{u-2}+1,\sigma_{u}]}$.

\begin{propos}\label{Cprop} For any $1\le u\le p$, the two-term object $C_u$ satisfies the following two conditions:
\begin{enumerate}
\item $\Hom^r(C_u,T_{\beta_u})=0$ for any $r\le \omega_u+\sigma_u$.
\item For any nontrivial two-term subobject $C'$ of $C_u$, there exists a nonzero morphism $f:C'\rightarrow T_{\beta_u}[r]$ with $\omega_u+\sigma_{u-2}+1\le r\le \omega_u+\sigma_u$ such that $f[\omega]g=0$ for any right socle morphism $g:C''\rightarrow C'[\omega]$.
\end{enumerate}
\end{propos}
\begin{proof} We proceed by induction on $u$ with the base case $u=0$. Since $C_0=P_i$ does not have nontrivial two-term subobjects and ${\Hom^r(P_i,T_\beta) = 0}$ for any $r\le m+\omega_0$ by Lemma \ref{base}, there is nothing left to prove in the base case. Now we prove $1)$ and $2)$ for $u$, assuming they are true for $u-1$. The first property of $C_u$ is clear, because 
$$\Hom^r(C_u,T_{\beta_u})=\Hom^r(t_{\Delta_u}^-C_{u-1},t_{\Delta_u}^{-1}T_{\beta_{u-1}})\cong \Hom^{r+\omega_{u+1}-1}(C_{u-1},T_{\beta_{u-1}}).$$

 \bigskip Now we turn to the second property. Suppose that $C'$ is a nontrivial two-term subobject of $C_u$. If $C'$ is right-proper, then $t_{\Delta_u}^+C'$ is a nontrivial two-term subobject of $C_{u-1}=t_{\Delta_u}^+C_u$ by Lemma \ref{presinc}. By the induction hypothesis, there is a nonzero morphism $f \colon t_{\Delta_u}^+C'\rightarrow T_{\beta_{u-1}}[r]$ with $\omega_{u-1}+\sigma_{u-3}+1\le r\le \omega_{u-1}+\sigma_{u-1}$ such that $f[\omega]g=0$ for any right socle morphism $g \colon C''\rightarrow t_{\Delta_u}^+C'[\omega]$.
 Let us prove that the morphism $t_{\Delta_u}^-f \colon C'\rightarrow T_{\beta_{u}}[r+1-\omega_{u-1}]$ satisfies the required properties.

 Suppose for a contradiction that there exists a right socle morphism ${g \colon C''\rightarrow C'[\omega]}$ such that ${(t_{\Delta_u}^-f)[\omega]g\not=0}$. Then applying Lemmas \ref{rightred} and \ref{simpvan} we can find such a $g$ with right-proper $C''$ satisfying the condition $\rs(C'')\subseteq \rs(C')$. Then ${t_{\Delta_u}^+g:t_{\Delta_u}^+C''\rightarrow t_{\Delta_u}^+C'[\omega]}$ is right socle by Lemma \ref{pressoc} and satisfies the condition $f[\omega]\big(t_{\Delta_u}^+g\big)=t_{\Delta_u}^+\left((t_{\Delta_u}^-f)[\omega]g\right)\not=0$ which is impossible. This shows that $t_{\Delta_u}^{-1}f$ is indeed the required morphism.

 It remains to consider the case when $C'$ is not right-proper. In this case we may assume that $C'=P_l$ for some $l\in\Delta_u$.
Let  $a$ be the minimal integer such that $P_l$ is a direct summand of $(t_lT_{\beta_u})_a$. Note that $a\in [\sigma_{u-3}+1,\sigma_{u-1}]$ by the properties \ref{nds}--\ref{ds} of $\Delta_u$.
Picking some long morphism $f \colon P_l\rightarrow t_lT_{\beta_u}[a+\omega]$ and applying $t_l^{-1}[1-\omega]$, we get a nonzero morphism $t_l^{-1}f[1-\omega] \colon P_l\rightarrow T_{\beta_u}[a+1]$. Note that 
$$a+1\in [\sigma_{u-3}+2,\sigma_{u-1}+1]=[\omega_u+\sigma_{u-2}+1,\omega_u+\sigma_u].$$
Thus, it remains to show that ${(t_l^{-1}f)[1]g=0}$ for any right socle morphism ${g \colon C''\rightarrow P_l[\omega]}$. Due to Lemma \ref{rightred}, we may assume that $C''$ has the form
$$
C''=cone\left(\bigoplus\limits_{j\in V^{u+1}}P_j^{x_j}[-\omega_{u+1}]\xrightarrow{\varphi_{C''}} P_l^x\right)
$$
for some integers $x_j,x$.  Note now that if $C''=P_l$, then $t_lg:P_l[1-\omega]\rightarrow P_l[1]$ is a right socle morphism, and hence factors though $P_j[1-\omega_{u+1}]$ for $j\in N(l)$. Then ${f[1](t_lg)=0}$ by the definition of a long morphism. Hence, we may assume that $C''$ is right-proper. Then $t_lg$ is a morphism from $cone\left(P_l^y[1-\omega]\xrightarrow{\varphi_{t_l^+C''}[1-\omega_{u+1}]}\bigoplus\limits_{j\in V^{u+1}}P_j^{x_j}[1-\omega_{u+1}]\right)$ to $P_l[1]$ by Lemma \ref{2toe}. Note that the composition
$$\bigoplus\limits_{j\in V^{u+1}}P_j^{x_j}[-\omega_{u+1}]\xrightarrow{\alpha_{t_l^+C''}[-\omega_{u+1}]} t_lC''[-1]\xrightarrow{t_lg[-1]}P_l\xrightarrow{f}t_lT_{\beta_u}[a+\omega]$$
is zero by the definition of a long morphism. Hence $f(t_lg)[-1]$ factors through some morphism ${\theta:P_l^y[1-\omega]\rightarrow t_lT_{\beta_u}[a+\omega]}$. Since $a+\omega-1<a$, $P_l$ is not a direct summand of $(t_lT_{\beta_u})_{a+\omega-1}$. If ${f[1](t_lg)\not=0}$, then $\theta\not=0$ and there is some $j\in N(l)$
such that $P_j\xrightarrow{\theta[\omega+\omega_u-1]\iota[\omega_u] \gamma_{j,l}} t_lT_{\beta_u}[a+2\omega+\omega_u-1]$ is nonzero, where $\iota:P_l\rightarrow P_l^y$ is some split monomorphism. Note that by property \ref{md} and the choice of $a$, the minimal nonzero degree of $t_lT_{\beta_u}$ is not smaller than $min(a, \sigma_{u-2}+1)$ by  Lemma \ref{tech0}. On the other hand, if ${f[1](t_lg)\not=0}$, then the minimal nonzero degree of $t_lT_{\beta_u}$ does not exceed ${a+\omega+\omega_u-1<a}$. Then we have
$$\sigma_{u-2}+1\le a+\omega+\omega_u-1\le \sigma_{u-1}+\omega+\omega_u-1=\sigma_{u-2}+\omega,$$ i.e. $\omega>0$, a contradiction. This shows that ${f[1](t_lg)=0}$, and hence $(t_l^{-1}f)[1]g=0$, as required.
\end{proof}

Now we are finially ready construct $\Delta_{p+1}\subseteq V_{p+1}$. We will do this in the following way. Set ${\Delta_{p+1}^0=\varnothing}$. Suppose that we have defined the set $\Delta_{p+1}^c$. Choose a pair $l, a$ with ${l\in V_{p+1}\setminus \Delta_{p+1}^c}$, $a \in \Z$ such that $P_l$ is a direct summand of $\left(t_{\Delta_{p+1}^c}^{-1}T_{\beta_p}\right)_a$, and with $a$ as small as possible (among all such pairs). We define $\Delta_{p+1}^{c+1}=\Delta_{p+1}^c\cup\{l\}$ and continue the process if $a\in [\sigma_{p-2}+1,\sigma_p]$. If either such an integer $a$ does not exist or $a>\sigma_p$, then we terminate the process, defining  $\Delta_{p+1}:=\Delta_{p+1}^c$. Now we are ready to prove the factorization theorem. 

\begin{theorem} The $\Delta_{p+1}$ constructed as described above satisfies the following conditions:
\begin{enumerate}
\item[1.] $\alpha = s_{\Delta_0}\dots s_{\Delta_{p+1}} \beta_{p+1}$ for some $\beta_{p+1}\in B_\Gamma^+$ with $l(\beta_{p+1})=l(\alpha)-\sum\limits_{v=0}^{p+1}|\Delta_v|$
\item[2.] $\Delta_{p-2}\subseteq\Delta_p\subseteq N(\Delta_{p-1})$.
\end{enumerate}
Moreover, if $\chi_{p+1}(l)>0$ for any $l\in\Delta_{p+1}$, then the following conditions are satisfied as well:
\begin{enumerate}
\item[3.] $P_l$ is not a direct summand of $(T_{\beta_{p+1}})_{[\sigma_{p-2}+1,\sigma_{p}]}$ for any $l \in V_{p+1}$.
\item[4.] The minimal nonzero degree of $T_{\beta_{p+1}}$ is not smaller than $\sigma_{p-1}+1$.
\item[5.] For any $l\in \Delta_{p+1}$, $P_l$ is a direct summand of $(t_lT_{\beta_{p+1}})_{[\sigma_{p-2}+1,\sigma_{p}]}$.
\end{enumerate}
\end{theorem}
\begin{proof} The first condition can be obtained applying Lemma \ref{princ} to the words $\beta_p$, $s_{\Delta_{p+1}^1}^{-1}\beta_p,\dots,s_{\Delta_{p+1}}^{-1}\beta_p$, all of which are of strictly smaller length than $\alpha$. Indeed, if $\Delta_{p+1}^{c+1}=\Delta_{p+1}^c\cup\{l\}$, then $P_l$ is a direct summand of $\left(t_{\Delta_{p+1}^c}^{-1}T_{\beta_p}\right)_a$ for $a\le \sigma_p$ and it is sufficient to prove that the minimal nonzero degree of $t_{\Delta_{p+1}^c}^{-1}T_{\beta_p}$ is not smaller than $a+\omega_p\le \sigma_{p-1}+1$. If this is not the case, there exists some $b\le min(a-1,\sigma_{p-1})$ and $k\in \Gamma_0$ such that $P_k$ is a direct summand of $(t_{\Delta_{p+1}^c}^{-1}T_{\beta_p})_b$. Observe that $k$ cannot belong to $V_p$ by the conditions \ref{nds}, \ref{md} with $u=p$ and  Lemma \ref{tech0}. $k$ also cannot belong to $V^{u+1}\setminus\Delta_{p+1}^c$ by the choice of $a$. On the other hand, for $k\in \Delta_{p+1}^c$, one has
$$\Hom^{b+\omega}(P_k,t_{\Delta_{p+1}^c}^{-1}T_{\beta_p})\cong \Hom^{b+\omega}(t_{\Delta_{p+1}^c}P_k,T_{\beta_p})\\
\cong\Hom^{b+2\omega-1}(P_k,T_{\beta_p})=0,$$
because the minimal nonzero degree of $T_{\beta_p}$ is not smaller than $\sigma_{p-2}+1$ and ${b+\omega-1\le \sigma_{p-1}+\omega-1\le \sigma_{p-2}}$.

 Let us now prove the second condition. Suppose first that there exists some ${l\in\Delta_{p+1}\setminus N(\Delta_p)}$. Then it is clear that $l\not\in N(\Delta_u)$ for all $0\le u\le p$, in particular, $l \neq i$. Hence we have $\alpha=s_{\Delta_0}\dots s_{\Delta_{p}} s_ls_l^{-1}\beta_p=s_l\gamma$ for some $\gamma\in B^+_{\Gamma}$ with $l(\gamma)=l(\alpha)-1$. Then the claim of Lemma \ref{princ} is valid for $\alpha$, a contradiction. 

 Suppose now that $l\not\in\Delta_{p+1}$ for some $l\in\Delta_{p-1}$. By construction, this means that $P_l$ is not a direct summand of $\left(t_{\Delta_{p+1}}^{-1}T_{\beta_p}\right)_{[\sigma_{p-2}+1,\sigma_p]}$. Let $\pi$ denote the split epimorphism $\bigoplus_{j \in \Delta_{p-1}} P_j^{\chi_{p-1}(j)}\twoheadrightarrow P_l^{\chi_{p-1}(l)}$. The octahedral axiom applied to the composition ${\pi[1-\omega_{p+1}]\circ \beta_{C_p}[1]}$ gives us the following diagram:
$$ {\small \xymatrix{ \bigoplus_{k \in \Delta_p} P_k^{\chi_p(k)} \ar[r] \ar@{=}[d]&
C'\ar[r]\ar[d] & \bigoplus_{j \in \Delta_{p-1}\setminus\{l\}} P_j^{\chi_{p-1}(j)}[1-\omega_{p+1}] \ar[d]\\
\bigoplus_{k \in \Delta_p} P_k^{\chi_p(k)} \ar[r] & C_p \ar[r]\ar[d]& \bigoplus_{j \in \Delta_{p-1}} P_j^{\chi_{p-1}(j)}[1-\omega_{p+1}] \ar[d]^{\pi[1-\omega_{p+1}]}\\ 
& P_l^{\chi_{p-1}(l)}[1-\omega_{p+1}]\ar@{=}[r] & P_l^{\chi_{p-1}(l)}[1-\omega_{p+1}] }
}$$
Let $\psi$ denote the morphism of $P_l^{\chi_{p-1}(l)}[-\omega_{p+1}]$ to $C'$ arising from this diagram.
Since $\chi_{p-1}(l)>0$ by our assumptions, $C'$ is a nontrivial two-term subobject of $C_p$ and Proposition \ref{Cprop} can be applied. For some $\omega_p+\sigma_{p-2}+1\le r\le \omega_p+\sigma_p$, we have a nonzero morphism $f:C'\rightarrow T_{\beta_p}[r]$ such that $f[\omega]g=0$ for any right socle morphism $g \colon C''\rightarrow C'[\omega]$. Since $\Hom^r(C_p,T_{\beta_p})=0$ by Proposition \ref{Cprop}, the morphism $f\psi'$ is nonzero for some component $\psi':P_l[-\omega_{p+1}]\rightarrow C'$ of the map $\psi$. We are going to prove that $t_{\Delta_{p+1}}^{-1}(f\psi')[\omega_{p+1}] \colon P_l\rightarrow t_{\Delta_{p+1}}^{-1}T_{\beta_p}[r+\omega_{p+1}]$ is long. Since ${r-\omega_p \in [\sigma_{p-2}+1,\sigma_p]}$, this contradicts the assumption that $P_l$ is not a direct summand of $\left(t_{\Delta_{p+1}}^{-1}T_{\beta_p}\right)_{[\sigma_{p-2}+1,\sigma_p]}$.


 Pick some $r\in N(l)$. Since $P_r[1-\omega_p]$ is a right-proper two-term object with $\alpha_{P_r[1-\omega_p]}=\varphi_{P_r[1-\omega_p]}=0$ and $\beta_{P_r[1-\omega_p]}=id_{P_r[\omega_p]}$, we have
$$t_{\Delta_{p+1}}^+P_r[1-\omega_p]\cong cone\left(\bigoplus\limits_{j\in N(r)\cap \Delta_{p+1}}P_j[-\omega_{p+1}]\rightarrow P_r\right)$$
by Lemma \ref{2toe}.
Note that the morphism $\psi'[\omega](t_{\Delta_{p+1}}^+\gamma_{r,l}[1-\omega_p])$ is right socle. Indeed, as the diagram above shows, $\psi'$ factors through $\bigoplus_{k \in \Delta_p} P_k^{\chi_p(k)}$ and the map
$$P_r\xrightarrow{\alpha_{t_{\Delta_{p+1}}^+P_r[1-\omega_p]}}t_{\Delta_{p+1}}^+P_r[1-\omega_p]\xrightarrow{t_{\Delta_{p+1}}^+\gamma_{r,l}[1-\omega_p]} P_l[\omega_p]\xrightarrow{}\bigoplus_{k \in \Delta_p} P_k^{\chi_p(k)}[\omega]$$
is not a split monomorphism, because it factors through $P_l[\omega_p]$. Then $(f\psi')[\omega](t_{\Delta_{p+1}}^+\gamma_{r,l}[1-\omega_p])=0$, and hence $t_{\Delta_{p+1}}^{-1}(f\psi')[\omega]\gamma_{r,l}=0$. Thus, $t_{\Delta_{p+1}}^{-1}(f\psi')[\omega_{p+1}]$ is a long morphism.

Suppose now that $\chi_{p+1}(l)>0$ for all $l\in\Delta_{p+1}$. We define $C_{p+1}:=t_{\Delta_{p+1}}^-C_p$. Then we have $\Hom^r(C_{p+1},T_{\beta_{p+1}})=0$ for $r\le \omega_{p+1}+\sigma_{p+1}$ (see the proof of the first part of Proposition \ref{Cprop}). Due to Lemma \ref{Cis2t}, we have $$C_{p+1}=cone\left(\bigoplus_{j \in \Delta_{p}} P_j^{\chi_{p}(j)}[-\omega_{p}] \xrightarrow{\varphi_{p+1}} \bigoplus_{k \in \Delta_{p+1}} P_k^{\chi_{p+1}(k)}\right)$$ for some morphism $\varphi_{p+1}$. We prove the third property by contradiction. If $P_l$ with $l\in\Delta_{p+1}$ is a direct summand of $(T_{\beta_{p+1}})_a$ for some $a\le \sigma_p$, then there exists a morphism of $\bigoplus_{k \in \Delta_{p+1}} P_k^{\chi_{p+1}(k)}$ to $T_{\beta_{p+1}}[a+\omega]$ annihilated by $\varphi_{p+1}$. This gives a nonzero morphism of $C_{p+1}$ to $T_{\beta_{p+1}}[a+\omega]$, which is impossible since $a+\omega\le \omega+\sigma_p<\omega_{p+1}+\sigma_{p+1}$. Observe that the minimal nonzero degree of $T_{\beta_{p+1}}$ is not smaller than $\sigma_{p-2}+1$ and $P_l$ with $l\in V_{p+1}\setminus \Delta_{p+1}$ cannot be a direct summand of $(T_{\beta_{p+1}})_{[\sigma_{p-2}+1,\sigma_p]}$ by the construction of $\Delta_{p+1}$. This finishes the proof of the third property. The forth property follows from the third one, the property \ref{nds} with $u=p$ and Lemma \ref{tech0}. 

 It remains to prove the fifth property. Let  ${l\in\Delta_{p+1}}$. There is some $c$ such that $l\not\in\Delta_{p+1}^{c}$ and $l\in\Delta_{p+1}^{c+1}$. Then $P_l$ is a direct summand of $\left(t_{\Delta_{p+1}^c}^{-1}T_{\beta_{p}}\right)_a$ for some $a\in [\sigma_{p-2}+1,\sigma_{p}]$. Choose the minimal such a number $a$. It follows from the construction of $\Delta_{p+1}$, the property \ref{nds} with $u=p$ and  Lemma \ref{tech0} that the minimal nonzero degree of $t_{\Delta_{p+1}^c}^{-1}T_{\beta_{p}}$ is not smaller than $min(a, \sigma_{p-1}+1)$ (see the beginning of this proof). Then $P_l$ is a direct summand of $\left(t_{\Delta_{p+1}\setminus\Delta_{p+1}^{c+1}}^{-1}t_{\Delta_{p+1}^c}^{-1}T_{\beta_{p}}\right)_a=(t_lT_{\beta_{p+1}})_a$ by Lemma \ref{tech0} and we are done.
\end{proof}

If $\chi_{p+1}(l) = 0$ for some $l \in \Delta_{p+1}$, we can write $\alpha = s_{\Delta_0} \cdots s_{\Delta_p} s_l \widetilde{\beta}$ and get the required presentation of $\alpha$, as we have announced in the previous section. In this case we terminate the process of factorization (step I) and go to braiding (step II), which is described in the next section. Otherwise \ref{taunz} is satisfied for $u = p+1$ as well. Then we write $\alpha = s_{\Delta_0} s_{\Delta_1}\cdots s_{\Delta_p} s_{\Delta_{p+1}} \beta_{p+1}$ and continue the process, applying the arguments above to $u=p+2$ instead of $u=p+1$.

\section{Mesh braiding in $B^+_\Gamma$}\label{br} 

The goal of this section is to show that any word of the form $s_{\Delta_0} \cdots s_{\Delta_p} s_l$ as obtained in the previous section is left-divisible by some $s_j$ in $B_\Gamma^+$ with $j\neq i$. As we explained in Section 5, this finishes the proof of Lemma \ref{princ}. 

\bigskip 
Let $\Z\Gamma$ be a directed graph whose vertices are pairs $(n,j)$ for every $n \in \Z$, $j \in V^n$. The arrows of $\Z\Gamma$ are of the form $((n,j) \to (n+1,i))$ for every edge $(i,j)$ of $\Gamma$. The graph $\Z\Gamma$ is essentially the stable translation quiver associated to $\Gamma$, but with vertices enumerated differently. It is convenient to draw $\Z\Gamma$ in such a way that vertices with the same first coordinate form one `vertical slice'. Set $p_1(a) = n$, $p_2(a) = j$ for $a=(n,j) \in (\Z\Gamma)_0$. 
Let $\Theta_n$ denote the set $p_1^{-1}(n) = \{(n,j) \in (\Z\Gamma)_0 \mid j\in V^n\} \subset (\Z\Gamma)_0$ (the $n$'th vertical slice).  


\bigskip

Now any word $\gamma$ representing an element of $ B_\Gamma^+$ can be depicted as some set of vertices $\Lambda_\gamma$ of $\Z\Gamma$ in the following way. Choose a presentation $\gamma = s_{\Sigma_0} \dots s_{\Sigma_p}$, where $\Sigma_0, \dots, \Sigma_p $ are such that $\Sigma_k \subseteq V^k$. Here we use the convention $s_{\varnothing} = 1$. Now we can define ${\Lambda_\gamma = \{(k, j) \mid k \geq 0, j \in \Sigma_k \} \subset (\Z\Gamma)_0}$. On the other hand, to any finite set $\Lambda \subset (\Z\Gamma)_0$ one can assign a word $\gamma_{\Lambda}$ in $B_\Gamma^+$. More precisely, set  ${\gamma_{\Lambda} = \prod\limits_{k = -\infty}^{+\infty} s_{\Sigma_k}}$, where $\Sigma_k = \Lambda \cap \Theta_k$.  It is easy to see that $\gamma_{\Lambda_\gamma}=\gamma$.

\bigskip 

  {\bf Example.}\label{ex1} Let $\Gamma = D_4$ with vertices enumerated in such a way that the vertex of degree three is labeled $2$. Let $\gamma = s_2(s_1s_3s_4)s_2(s_1s_3s_4)s_2s_4$. Then we have the following set $\Lambda_{\gamma}$:

\begin{figure}[H]
\centering

\begin{tikzpicture}

\draw (-3,1.5) -- (-2,0.5) -- (-1,1.5) -- (0,0.5) --(1,1.5);
\draw  (-3,0.5) --  (1,0.5);
\draw (-3,-0.5) --  (-2,0.5) --  (-1,-0.5) --(0,0.5); 
\draw  (-4,0.5) --(-3,1.5); 
\draw  (-4,0.5) -- (-3,0.5); 
\draw  (-4,0.5) -- (-3,-0.5); 
\draw (0,0.5) --  (1,-0.5);

\draw [fill=white,white]  (-3,-0.5) circle [radius=0.2];
\draw [fill=white,white]  (-3,0.5) circle [radius=0.2];
\draw [fill=white,white]  (-3,1.5) circle [radius=0.2];
\draw [fill=white,white]  (-2,0.5) circle [radius=0.2];
\draw [fill=white,white]  (-1,-0.5) circle [radius=0.2];
\draw [fill=white,white]  (-1,0.5) circle [radius=0.2];
\draw [fill=white,white]  (-1,1.5) circle [radius=0.2];
\draw [fill=white,white]  (0,0.5) circle [radius=0.2];
\draw [fill=white,white]  (0,0.5) circle [radius=0.2];
\draw [fill=white,white] (1,-0.5) circle [radius=0.2];
\draw [fill=white,white]   (1,0.5)  circle [radius=0.2];
\draw [fill=white,white]   (1,1.5)  circle [radius=0.2];
\draw [fill=white,white]  (-4,0.5)  circle [radius=0.2];
\draw [fill=red,red]  (-3,0.5) circle [radius=0.05];
\draw [fill=red,red]   (-3,1.5) circle [radius=0.05];
\draw [fill=red,red]  (-1,-0.5) circle [radius=0.05];
\draw [fill=red,red]   (-1,0.5) circle [radius=0.05];
\draw [fill=red,red]  (-1,1.5) circle [radius=0.05];
\draw [fill=red,red]  (0,0.5) circle [radius=0.05];
\draw [fill=red,red]   (0,0.5) circle [radius=0.05];
\draw [fill=red,red]   (-4,0.5) circle [radius=0.05];
\draw [fill=red,red]   (-3,-0.5) circle [radius=0.05];
\draw [fill=red,red]   (-2,0.5)  circle [radius=0.05];
\draw [fill=red,red]  (1,1.5)  circle [radius=0.05];
\draw [fill=black,black]   (1,0.5)  circle [radius=0.05];
\draw [fill=black,black]  (1,-0.5) circle [radius=0.05];

\node at (-3,-0.5) {};
\node at (-3,0.5) {};
\node at (-3,1.5) {};
\node at (-2,0.5) {};

\node at (-1,-0.5) {};
\node at (-1,0.5) {};
\node at (-1,1.5) {};
\node at (0,0.5) {};
\node at (1,-0.5) {};
\node at (1,0.5) {};
\node at (1,1.5) {};
\node at (-6.5,-1.5) {\footnotesize $p_1$};
\node [slice] at (-3,-1.5) {\footnotesize 1};
\node [slice] at (-2,-1.5) {\footnotesize 2};
\node [slice] at (-1,-1.5) {\footnotesize 3};
\node [slice] at (0,-1.5) {\footnotesize 4};
\node [slice] at (1,-1.5) {\footnotesize 5};
\node [slice] at (-4,-1.5) {\footnotesize 0};

\node at (-4,0.5) {};
\end{tikzpicture}

\caption{}

\end{figure}

On the other hand, observe that there are infinitely many sets $\Lambda$ such that $\gamma_{\Lambda} =  \gamma$. For instance, the following set of vertices of $(\Z\Gamma)_0$ also corresponds to the word $\gamma$:   

\begin{figure}[H]
\centering
\begin{tikzpicture}

\draw (-3,1.5) -- (-2,0.5) -- (-1,1.5) -- (0,0.5) --(1,1.5);
\draw  (-3,0.5) --  (1,0.5);
\draw (-3,-0.5) --  (-2,0.5) --  (-1,-0.5) --(0,0.5); 
\draw  (-4,0.5) --(-3,1.5); 
\draw  (-4,0.5) -- (-3,0.5); 
\draw  (-4,0.5) -- (-3,-0.5); 
\draw (0,0.5) --  (1,-0.5);
\draw (3,1.5) -- (2,0.5) -- (3,0.5);
\draw (3,-0.5)  -- (2, 0.5);
\draw (1,1.5) --  (2,0.5) -- (1,0.5);
\draw (1,-0.5) -- (2,0.5);
\draw [fill=white,white]  (2,0.5) circle [radius=0.2];
\draw [fill=white,white]  (3,1.5)  circle [radius=0.2];
\draw [fill=white,white]  (3,0.5)  circle [radius=0.2];
\draw [fill=white,white]  (3,-0.5) circle [radius=0.2];
\draw [fill=white,white]  (-3,-0.5) circle [radius=0.2];
\draw [fill=white,white]  (-3,0.5) circle [radius=0.2];
\draw [fill=white,white]  (-3,1.5) circle [radius=0.2];
\draw [fill=white,white]  (-2,0.5) circle [radius=0.2];
\draw [fill=white,white]  (-1,-0.5) circle [radius=0.2];
\draw [fill=white,white]  (-1,0.5) circle [radius=0.2];
\draw [fill=white,white]  (-1,1.5) circle [radius=0.2];
\draw [fill=white,white]  (0,0.5) circle [radius=0.2];
\draw [fill=white,white]  (0,0.5) circle [radius=0.2];
\draw [fill=white,white] (1,-0.5) circle [radius=0.2];
\draw [fill=white,white]   (1,0.5)  circle [radius=0.2];
\draw [fill=white,white]   (1,1.5)  circle [radius=0.2];
\draw [fill=white,white]  (-4,0.5)  circle [radius=0.2];
\draw [fill=red,red]  (-3,0.5) circle [radius=0.05];
\draw [fill=red,red]   (-3,1.5) circle [radius=0.05];
\draw [fill=red,red]  (-1,-0.5) circle [radius=0.05];
\draw [fill=black,black]   (-1,0.5) circle [radius=0.05];
\draw  [fill=black,black]   (-1,1.5) circle [radius=0.05];
\draw [fill=black,black]   (0,0.5) circle [radius=0.05];
\draw [fill=red,red]   (-4,0.5) circle [radius=0.05];
\draw [fill=red,red]   (-3,-0.5) circle [radius=0.05];
\draw [fill=red,red]   (-2,0.5)  circle [radius=0.05];
\draw [fill=red,red]  (1,1.5)  circle [radius=0.05];
\draw [fill=red,red]  (2,0.5)  circle [radius=0.05];
\draw [fill=red,red]  (3,1.5)   circle [radius=0.05];
\draw [fill=red,red]  (1,0.5)  circle [radius=0.05];
\draw [fill=black,black]  (1,-0.5) circle [radius=0.05];
\draw [fill=black,black]  (3,0.5)  circle [radius=0.05];
\draw [fill=black,black]  (3,-0.5) circle [radius=0.05];

\node at (-3,-0.5) {};
\node at (-3,0.5) {};
\node at (-3,1.5) {};
\node at (-2,0.5) {};

\node at (-1,-0.5) {};
\node at (-1,0.5) {};
\node at (-1,1.5) {};
\node at (0,0.5) {};
\node at (1,-0.5) {};
\node at (1,0.5) {};
\node at (1,1.5) {};
\node at (-6.5,-1.5) {\footnotesize $p_1$};
\node [slice] at (-3,-1.5) {\footnotesize 1};
\node [slice] at (-2,-1.5) {\footnotesize 2};
\node [slice] at (-1,-1.5) {\footnotesize 3};
\node [slice] at (0,-1.5) {\footnotesize 4};
\node [slice] at (1,-1.5) {\footnotesize 5};
\node [slice] at (-4,-1.5) {\footnotesize 0};
\node [slice] at (2,-1.5) {\footnotesize 6};
\node [slice]  at (3,-1.5) {\footnotesize 7};

\node at (-4,0.5) {};
\node at (2,0.5) {};
\node at (3,1.5) {};
\node at (3,0.5) {};
\node at (3,-0.5) {};

\end{tikzpicture}
\caption{}
\end{figure}

Now fix some finite $\Lambda \subset (\Z\Gamma)_0$. 

\begin{definition}
Let $a,b \in \Lambda$. We say that there is a {\it generalized mesh}  starting at $a$ and ending at $b$ if $p_2(a)=p_2(b)=t$, $p_1(a)<p_1(b)$ and $(k,t) \notin \Lambda$ for every $k$ with $p_1(a) < k < p_1(b)$. 
In this case we set
$${mesh_{\Lambda}(a,b) = \{c\in \Lambda \mid p_2(c) \in N(t), p_1(a) < p_1(c) < p_1(b) \}}.$$ 
We will sometimes refer to $mesh_{\Lambda}(a,b)$ as a generalized mesh as well. 

\bigskip 
For every $b \in \Lambda$, let $\tau_{{\Lambda}}(b)$ denote the element of $\Lambda$ such that there is a generalized mesh starting at $\tau_{{\Lambda}}(b)$ and ending at $b$. If there is no such an element in $\Lambda$, we say that $\tau_{{\Lambda}}(b)$ is an `imaginary' vertex $(-\infty,p_2(b))$. In this case we also say that there is an {\it infinite mesh} starting at $\tau_{{\Lambda}}(b)$ and ending at $b$ and set ${mesh_{\Lambda}(\tau_{{\Lambda}}(b),b) = \{c\in\Lambda \mid p_2(c) \in N(p_2(b)), p_1(c) < p_1(b) \}}$.

\end{definition}

Now suppose there is a non-negative integer assigned to each element of $\Lambda$ and each of the imaginary vertices at minus infinity, i.e. consider $\Lambda$ together with a function ${\theta \colon \Lambda \sqcup \big(\{-\infty\}\times \Gamma_0\big) \to \Z}$.

\begin{definition} We say that $(\Lambda, \theta)$ {\it satisfies mesh relations} if 
$$\theta(b) + \theta\tau_{\Lambda}({b}) = \sum\limits_{c \in mesh_{\Lambda}(\tau_{\Lambda}({b}),b)} \theta(c)$$ 
for every $b \in \Lambda$ \end{definition} 

The following statement follows immediately from the definitions: 

\begin{lemma}\label{in_func} Let $\gamma$ be a word of the form $s_{\Delta_0}\dots s_{\Delta_p} s_l$ as described in the previous section. For every $(n,j) \in \Lambda_\gamma$ let $\theta(n,j) = \chi_n(j)$, $\theta(-\infty,i) = -1$, $\theta(-\infty,j) = 0$ for every $j \neq i$. Then $(\Lambda_\gamma, \theta)$ satisfies mesh relations. \end{lemma}

  {\bf Example:} The word  $s_2(s_1s_3s_4)s_2(s_1s_3s_4)s_2s_4 $ in $B_{D_4}^+$ as in the previous example is depicted below together with the function $\theta = \chi$. 

\begin{figure}[H]
\centering 
\begin{tikzpicture}
\draw (-3,1.5) -- (-2,0.5) -- (-1,1.5) -- (0,0.5) --(1,1.5);
\draw  (-3,0.5) --  (1,0.5);
\draw (-3,-0.5) --  (-2,0.5) --  (-1,-0.5) --(0,0.5); 
\draw  (-4,0.5) --(-3,1.5); 
\draw  (-4,0.5) -- (-3,0.5); 
\draw  (-4,0.5) -- (-3,-0.5); 
\draw (0,0.5) --  (1,-0.5);

\draw [fill=white,white]  (-3,-0.5) circle [radius=0.2];
\draw [fill=white,white]  (-3,0.5) circle [radius=0.2];
\draw [fill=white,white]  (-3,1.5) circle [radius=0.2];
\draw [fill=white,white]  (-2,0.5) circle [radius=0.2];
\draw [fill=white,white]  (-1,-0.5) circle [radius=0.2];
\draw [fill=white,white]  (-1,0.5) circle [radius=0.2];
\draw [fill=white,white]  (-1,1.5) circle [radius=0.2];
\draw [fill=white,white]  (0,0.5) circle [radius=0.2];
\draw [fill=white,white]  (0,0.5) circle [radius=0.2];
\draw [fill=white,white] (1,-0.5) circle [radius=0.2];
\draw [fill=white,white]   (1,0.5)  circle [radius=0.2];
\draw [fill=white,white]   (1,1.5)  circle [radius=0.2];
\draw [fill=white,white]  (-4,0.5)  circle [radius=0.2];
\draw [fill=black,black]   (1,0.5)  circle [radius=0.05];
\draw [fill=black,black]  (1,-0.5) circle [radius=0.05];

\node at (-3,-0.5)  {\footnotesize 1};
\node at (-3,0.5)  {\footnotesize 1};
\node at (-3,1.5)  {\footnotesize 1};
\node at (-2,0.5)  {\footnotesize 2};

\node at (-1,-0.5)  {\footnotesize 1};
\node at (-1,0.5)  {\footnotesize 1};
\node at (-1,1.5)  {\footnotesize 1};
\node at (0,0.5)  {\footnotesize 1};
\node at (1,-0.5) {};
\node at (1,0.5) {};
\node at (1,1.5)  {\footnotesize 0};
\node at (-6.5,-1.5) {\footnotesize $p_1$};
\node [slice] at (-3,-1.5) {\footnotesize 1};
\node [slice] at (-2,-1.5) {\footnotesize 2};
\node [slice] at (-1,-1.5) {\footnotesize 3};
\node [slice] at (0,-1.5) {\footnotesize 4};
\node [slice] at (1,-1.5) {\footnotesize 5};
\node [slice] at (-4,-1.5) {\footnotesize 0};

\node at (-4,0.5) {\footnotesize 1};
\end{tikzpicture}
\caption{}
\end{figure} 

It is reasonable to ask which finite subsets of $(\Z\Gamma)_0$ correspond to words that are equal in the braid monoid $B_\Gamma^+$. Two words representing the same element of $B_\Gamma^+$ can be obtained one from another by applying braid and commutation relations. Hence, we would quite naturally expect to have some analogous transformations on the `$\Z \Gamma$ side' of the story, leaving the corresponding element of the braid monoid unchanged. We will introduce these transformations in a moment.

 First let $a = (n,j) \in \Lambda$ be such that ${b = (n+2,j) \notin \Lambda}$ (respectively ${b = (n-2, j) \notin \Lambda}$) and $(n+1,k) \notin \Lambda$ (respectively $(n-1,k) \notin \Lambda$) for every $k \in N(j)$. Set $\overline{\Lambda} = \big(\Lambda \setminus \{a\} \big) \cup \{b\} $. In addition, let $\overline{\theta}(b) = \theta(a)$ and $\overline{\theta}|_{\overline{\Lambda}\setminus\{b\}} = \theta|_{\Lambda\setminus\{a\}}$.  We refer to this procedure as commutation.
 
 \begin{figure}[H]\centering
\begin{tikzpicture}
\draw (-2.5,0.5) --  (-1.5,0.5)  -- (-0.5,0.5) -- (-1.5,1.5) --  (-2.5,0.5); 
\draw (-2.5,0.5) --  (-1.5, -0.5) -- (-0.5,0.5);
\draw (1.5,0.5) -- (2.5,0.5) --  (3.5,0.5) -- (2.5,1.5) -- (1.5,0.5) -- (2.5,-0.5) -- (3.5,0.5);

\draw [fill=white,white] (-2.5,0.5)  circle [radius=0.3];
\draw [fill=white,white] (-1.5,0.5) circle [radius=0.2];
\draw [fill=white,white] (-0.5,0.5)  circle [radius=0.2];
\draw [fill=white,white] (-1.5,1.5)  circle [radius=0.2];
\draw [fill=white,white] (-1.5,-0.5) circle [radius=0.2];

\draw [fill=white,white] (1.5,0.5)  circle [radius=0.2];
\draw [fill=white,white] (2.5,0.5) circle [radius=0.2];
\draw [fill=white,white] (3.5,0.5)  circle [radius=0.3];
\draw [fill=white,white] (2.5,1.5) circle [radius=0.2];
\draw [fill=white,white] (2.5,-0.5) circle [radius=0.2];
\draw [fill=black,black]  (-1.5,0.5) circle [radius=0.05];
\draw [fill=black,black]  (-0.5,0.5) circle [radius=0.05];
\draw [fill=black,black]  (-1.5,1.5) circle [radius=0.05];
\draw [fill=black,black]  (-1.5,-0.5) circle [radius=0.05];
\draw [fill=black,black]  (1.5,0.5) circle [radius=0.05];
\draw [fill=black,black]  (2.5,0.5) circle [radius=0.05];
\draw [fill=black,black]  (2.5,1.5) circle [radius=0.05];
\draw [fill=black,black]  (2.5,-0.5) circle [radius=0.05];

\node at (-2.5,0.5) {\footnotesize $\theta(a)$};
\node at (-1.5,0.5) {};
\node at (-0.5,0.5) {};
\node at (-1.5,1.5) {};
\node at (-1.5,-0.5) {};

\node at (1.5,0.5) {};
\node at (2.5,0.5) {};
\node at (3.5,0.5) {\footnotesize $\theta(a)$};
\node at (2.5,1.5) {};
\node at (2.5,-0.5) {};
\node at (-4.5,-1.5)   {\footnotesize $p_1$};
\node at (-2.5,-1.5)  [slice]{\footnotesize $n$};
\node at (-1.5,-1.5)  [slice] {\footnotesize $n+1$};
\node at (-0.5,-1.5)  [slice] {\footnotesize $n+2$};

\node [slice] at (1.5,-1.5)  {\footnotesize $n$};
\node [slice] at (2.5,-1.5)  {\footnotesize $n+1$};
\node [slice] at (3.5,-1.5) {\footnotesize $n+2$};
\end{tikzpicture}
\caption{}

\end{figure}

The following statement is clear.

 \begin{lemma} Commutation does not change the corresponding word in $B_\Gamma^+$, i.e. $\gamma_{\Lambda} = \gamma_{\overline{\Lambda}}$. In addition, if $(\Lambda, \theta)$ satisfies mesh relations, then so does  $(\overline{\Lambda}, \overline{\theta})$.
 \end{lemma}
 
 Now braid relations in $B_\Gamma^+$ can be also depicted in terms of vertices of $\Z\Gamma$. We say that $a,b,c \in (\Z\Gamma)_0$ form {\it a braid} if $a = (n,j)$, $b = (n+1,k)$, $c = (n+2,j)$ for some $n \in \Z$ and some $j,k \in \Gamma_0$ such that $k \in N(j)$. Now let $a,b,c \in \Lambda$ as above form a braid and suppose in addition that $(n+1,t) \notin \Lambda$ for any ${t \in N(j) \setminus \{k\}}$, $(n+2,l) \notin \Lambda$ for any ${l \in N(k) \setminus \{j\}}$ and $d = (n+3, k) \notin \Lambda$. Set $\overline{\Lambda} = (\Lambda \setminus \{a \}) \cup  \{d\}$,  $\overline{\theta}|_{\overline{\Lambda} \setminus \{b,c,d\}} = \theta|_{\Lambda \setminus \{a,b,c\}}$, $\overline{\theta}(b) = \theta(c), \overline{\theta}(c) = \theta(b), \overline{\theta}(d) = \theta(a)$. We refer to this procedure as braiding. 
 
 \begin{figure}[H]
\centering
\begin{tikzpicture}
\draw (-2.5,0.5) --  (-1.5,0.5)  -- (-0.5,0.5) -- (-1.5,1.5) --  (-2.5,0.5); 
\draw (-2.5,0.5) --  (-1.5, -0.5) -- (-0.5,0.5);
\draw (3.5,0.5) -- (4.5,0.5) --  (5.5,0.5) -- (4.5,1.5) -- (3.5,0.5) -- (4.5,-0.5) -- (5.5,0.5);
\draw (-1.5,1.5) --  (-0.5,2.5) -- (0.5,1.5) -- (-0.5,0.5); 
\draw (4.5,1.5) --  (5.5,2.5) -- (6.5,1.5) -- (5.5,0.5); 

\draw [fill=white,white] (-2.5,0.5)  circle [radius=0.3];
\draw [fill=white,white] (-1.5,0.5) circle [radius=0.2];
\draw [fill=white,white] (-0.5,0.5)  circle [radius=0.3];
\draw [fill=white,white] (-1.5,1.5)  circle [radius=0.3];
\draw [fill=white,white] (-1.5,-0.5) circle [radius=0.2];
\draw [fill=white,white] (3.5,0.5)  circle [radius=0.2];
\draw [fill=white,white] (4.5,0.5) circle [radius=0.3];
\draw [fill=white,white] (3.5,0.5)  circle [radius=0.2];
\draw [fill=white,white] (4.5,1.5) circle [radius=0.2];
\draw [fill=white,white] (4.5,-0.5) circle [radius=0.2];
\draw [fill=white,white] (5.5,2.5) circle [radius=0.2];
\draw [fill=white,white]  (6.5,1.5) circle [radius=0.2];
\draw [fill=white,white]  (0.5,1.5) circle [radius=0.2];
\draw [fill=white,white]  (5.5,0.5) circle [radius=0.3];
\draw [fill=white,white]  (-0.5,2.5) circle [radius=0.2];
draw [fill=black,black]  (-1.5,0.5) circle [radius=0.05];
\draw [fill=black,black]  (-1.5,-0.5) circle [radius=0.05];
\draw [fill=black,black]  (3.5,0.5) circle [radius=0.05];
\draw [fill=black,black]  (4.5,0.5) circle [radius=0.05];
\draw [fill=black,black]  (4.5,-0.5) circle [radius=0.05];
\draw [fill=black,black]  (5.5,2.5)  circle [radius=0.05];
\draw [fill=black,black]   (0.5,1.5) circle [radius=0.05];
\draw [fill=black,black]  (-0.5,2.5) circle [radius=0.05];
\draw [fill=black,black] (-1.5,0.5) circle [radius=0.05];

\node at (-2.5,0.5) {\footnotesize $\theta(a)$};
\node at (-1.5,0.5) {};
\node at (-0.5,0.5) {\footnotesize $\theta(c)$};
\node at (-1.5,1.5) {\footnotesize $\theta(b)$};
\node at (-1.5,-0.5) {};

\node at (3.5,0.5) {};
\node at (4.5,0.5) {};
\node at (5.5,0.5) {\footnotesize $\theta(b)$};
\node at (4.5,1.5) {\footnotesize $\theta(c)$};
\node at (4.5,-0.5) {};
\node at (-4.5,-1.5)   {\footnotesize $p_1$};
\node at (-2.5,-1.5)  [slice]{\footnotesize $n$};
\node at (-1.5,-1.5)  [slice] {\footnotesize $n+1$};
\node at (-0.5,-1.5)  [slice] {\footnotesize $n+2$};
\node at (0.5,-1.5)  [slice] {\footnotesize $n+3$};

\node [slice] at (3.5,-1.5)  {\footnotesize $n$};
\node [slice] at (4.5,-1.5)  {\footnotesize $n+1$};
\node [slice] at (5.5,-1.5) {\footnotesize $n+2$};
\node [slice] at (6.5,-1.5) {\footnotesize $n+3$};
\node at (-0.5,2.5) {};
\node at (0.5,1.5) {};
\node at (6.5,1.5) {\footnotesize $\theta(a)$};
\node at (5.5,2.5) {};
\end{tikzpicture}
\caption{}
\end{figure} 
 
 \begin{lemma} Braiding does not change the corresponding word in $B_{\Gamma}^+$, i.e. $\gamma_\Lambda = \gamma_{\overline{\Lambda}}$. Moreover, if $(\Lambda, \theta)$ satisfies mesh relations, then so does $(\overline{\Lambda}, \overline{\theta})$.  \end{lemma} 
 
 \begin{proof}
 First we establish that $\gamma_\Lambda = \gamma_{\overline{\Lambda}}$. It is sufficient to show that $s_{\Sigma_n}s_{\Sigma_{n+1}}s_{\Sigma_{n+2}}s_{\Sigma_{n+3}} = s_{\overline{\Sigma}_n}s_{\overline{\Sigma}_{n+1}}s_{\overline{\Sigma}_{n+2}}s_{\overline{\Sigma}_{n+3}}$ where $\Sigma_p=p_2\left(\Theta_p \cap \Lambda\right)$ and $\overline{\Sigma}_p=p_2\left(\Theta_p \cap \overline{\Lambda}\right)$. Observe that $\overline{\Sigma}_n = \Sigma_n \setminus \{j\}$, $\overline{\Sigma}_{n+1} = \Sigma_{n+1}$, $\overline{\Sigma}_{n+2}= \Sigma_{n+2}$ and $\overline{\Sigma}_{n+3} = \Sigma_{n+3} \cup \{k\}$. Hence 
 
 $$s_{\Sigma_n}s_{\Sigma_{n+1}}s_{\Sigma_{n+2}}s_{\Sigma_{n+3}} = s_{\overline{\Sigma}_n}s_js_{\Sigma_{n+1} \setminus \{k\}} s_k s_j s_{\Sigma_{n+2} \setminus \{j\}}s_{\Sigma_{n+3}} =  $$ $$ s_{\overline{\Sigma}_n}s_{\Sigma_{n+1} \setminus \{k\}} s_j s_k s_j s_{\Sigma_{n+2} \setminus \{j\}}s_{\Sigma_{n+3}} =  s_{\overline{\Sigma}_n}s_{\Sigma_{n+1} \setminus \{k\}} s_k s_j s_k s_{\Sigma_{n+2} \setminus \{j\}} s_{\Sigma_{n+3}} = $$ 
 $$s_{\overline{\Sigma}_n}s_{\overline{\Sigma}_{n+1}} s_j  s_{\Sigma_{n+2} \setminus \{j\}} s_k s_{\Sigma_{n+3}} = s_{\overline{\Sigma}_n}s_{\overline{\Sigma}_{n+1}}s_{\overline{\Sigma}_{n+2}}s_{\overline{\Sigma}_{n+3}} $$
 
 Now we will show that braiding respects mesh relations. Since $\overline{\theta}$ differs from $\theta$ only on $a,b,c$ and $d$, it is sufficient to check the relations only for generalized meshes these four vertices take part in. There are two types of such generalized meshes: those that have one of $a,b,c$ and $d$ as starting and/or ending vertices and those that do not. 
 
 \begin{enumerate}
     \item Note that $\tau_{\overline{\Lambda}}(b)=\tau_{{\Lambda}}(b)$ and $mesh_{\overline{\Lambda}}(\tau_{\overline{\Lambda}}(b),b) = mesh_{\Lambda}(\tau_{{\Lambda}}(b),b) \setminus \{a\}$. We have 
     $$\theta\tau_{{\Lambda}}(b) + \theta(b) = \sum_{x \in mesh_{\Lambda}(\tau_{{\Lambda}}(b), b)} \theta(x)  = \sum_{x \in mesh_{\overline{\Lambda}}(\tau_{\overline{\Lambda}}(b),b)} \theta(x)  + \theta(a)$$ 
     and $\theta(a) + \theta(c) = \theta(b)$, because $a,b,c$ form a generlized mesh in $\Lambda$ starting at $a$ and ending at $c$. Then 
        $$\overline{\theta}\tau_{\overline{\Lambda}}(b) + \overline{\theta}(b) = \theta\tau_{{\Lambda}}(b) + \theta(c) = \sum_{ x\in mesh_{\overline{\Lambda}}(\tau_{\overline{\Lambda}}(b),b)} \theta(x).$$ 
        The remaining cases are either analogous to the one just discussed or obvious.

     \item Now consider some generalized mesh $mesh_{\overline{\Lambda}}(x,y)$ in $\overline{\Lambda}$ such that $x,y\not\in \{a,b,c,d\}$. If $b \in mesh_{\overline{\Lambda}}(x,y)$, then $d \in mesh_{\overline{\Lambda}}(x,y)$ and vise versa. The corresponding mesh relation remains valid after braiding, because $d \notin \Lambda$ and $\overline{\theta}(b) + \overline{\theta}(d) = \theta(c) + \theta(a) = \theta(b)$. If $c \in mesh_{\overline{\Lambda}}(x,y)$, then $a \in mesh_{\Lambda}(x,y)$ and vise versa. The corresponding mesh relation remains valid after braiding, because $\overline{\theta}(c) = \theta(b) = \theta(c) + \theta(a)$. 
 \end{enumerate}
 \end{proof}
 
 Now we are fully equipped to state the main result of this section. As we have explained earlier, the following theorem applied to $(\Lambda, \theta)$ as in Lemma \ref{in_func} finishes the proof of Lemma \ref{princ}.
 
 \begin{theorem}\label{mesh_thm} Let $(\Lambda, \theta)$ satisfying mesh relations be such that $\theta$ is non-negative on $\Lambda$, vanishing on precisely one vertex of $\Lambda$, $\theta(-\infty,i) = -1$ for some $i \in \Gamma_0$ and $\theta(-\infty,k) = 0$ for all $k \in \Gamma_0\setminus \{i\}$. Then $\gamma = \gamma_{\Lambda}$ is left-divisible by some $s_j$ in $B_\Gamma^+$ with $j\neq i$.  \end{theorem}
 
 To prove this theorem we will first need the following lemma.
 
 \begin{lemma}\label{mesh} Let $(\Lambda, \theta)$ be as in Theorem \ref{mesh_thm} and $a \in \Lambda$ such that $\theta(a) = 0$, $\tau_{\Lambda}(a)\not=-\infty$.
 Then there exists a sequence of commutations and braidings turning $(\Lambda, \theta)$ into some $(\Lambda', \theta')$ such that
 $$|\{x \in \Lambda' \mid p_1(x) \leq p_1(b)  \}|<|\{x \in \Lambda \mid p_1(x) \leq p_1(a)  \}|,$$
 where $b$ is the unique vertex of $\Lambda'$ on which $\theta'$ vanishes.
 \end{lemma}
 
 Let us first show how Lemma \ref{mesh} implies Theorem \ref{mesh_thm}.
 
 \begin{proof}[Proof of Theorem \ref{mesh_thm}] It is sufficient to show that applying a sequence of commutations and braidings we can transform $\Lambda$ into some $\Lambda'$ such that $\Theta_n \cap \Lambda' \neq \{i\}$, where $n$ is the smallest integer such that $\Theta_n \cap \Lambda' \neq \varnothing$. Note that commutation and braiding do not change the multiset  $\{\theta(x)\}_{x \in \Lambda}$ of values of $\theta$.
 
 Applying Lemma \ref{mesh} several times, one can obtain a pair $(\Lambda'', \theta'')$ satisfying mesh relations such that $\gamma_{\Lambda''} = \gamma_{\Lambda}$ and $\theta''$ vanishes on (and only on) a vertex with the smallest first coordinate among all vertices of $\Lambda''$. Let $\theta''(x) = 0$. Then $\gamma_{\Lambda''} = \gamma_{\Lambda}$ is divisible by $s_{p_2(x)}$ on the left. It remains to ascertain that ${p_2(x) \neq i}$. Indeed, since $x$ is a vertex  with the smallest $p_1(x)$ among all vertices in $\Lambda''$, there is an infinite generalized mesh ending at $x$ and starting at $(-\infty, j)$ for some $j \in \Gamma_0$. Since $(\Lambda'', \theta'')$ satisfies infinite mesh relations, $0 = \theta''(-\infty, j) + \theta''(x) = \theta''(-\infty, j)$. We immediately see that $j \neq i$, since $\theta''(-\infty, i) = -1$.
 \end{proof} 
 
  Now we ready to prove Lemma \ref{mesh}.
 
 \begin{proof}[Proof of Lemma \ref{mesh}]
 Let $a_0 = a = (n,j)$.
 Since there is only one vertex of $\Lambda$ on which $\theta$ vanishes and $\tau_{\Lambda}(a_0)\in\Lambda$, $\theta\tau_{\Lambda}(a_0) + \theta(a_0) = \theta\tau_{\Lambda}(a_0) > 0$. Hence ${mesh_{\Lambda}(\tau_{\Lambda}(a_0),a_0) \neq \varnothing}$. Let ${a_1 = (n_1,t) \in mesh_{\Lambda}(\tau_{\Lambda}(a_0),a_0)}$ be any vertex with the smallest first coordinate among vertices of ${mesh_{\Lambda}(\tau_{\Lambda}(a_0),a_0)}$. Then, possibly applying several commutations, one can assume that $\tau_{\Lambda}(a_0) = (n_1-1,j)$. If ${|mesh_{\Lambda}(\tau_{\Lambda}(a_0),a_0)| = 1}$, then, also possibly applying several commutations, one can assume that $\tau_{\Lambda}(a_0),a_1,a_0$ form a braid. Then we can apply braiding to the triple $(\tau_{\Lambda}(a_0),a_1,a_0)$ after some commutations as described below and we are done. Suppose now that ${q_0 =|mesh_{\Lambda}(\tau_{\Lambda}(a_0),a_0)| > 1}$.

 Note that $q_1 = |mesh_{\Lambda}(\tau_{\Lambda}(a_1),a_1)| > 0$, since $\tau_{\Lambda}(a_0) \in mesh_{\Lambda}(\tau_{\Lambda}(a_1),a_1)$. Suppose that $q_1>1$.  Let $a_2 = (n_2,p) \in mesh_{\Lambda}(\tau_{\Lambda}(a_1), a_1)$ be any vertex with the smallest first coordinate among vertices of $mesh_{\Lambda}(\tau_{\Lambda}(a_1),a_1)\setminus\{\tau_{\Lambda}(a_0)\}$. Then, possibly applying several commutations, one can assume that $\tau_{\Lambda}(a_1) = (n_2-1,t)$.
 
 \begin{figure}[H]
  \centering 
 \begin{tikzpicture}
 \draw (-0.5,1) -- (0.5,0); 
 \draw [dotted] (0.5,0) -- (3.5,0);
 \draw (3.5,0) -- (4.5,1);
 \draw (4.5,1) -- (3.5,2);
 \draw (-0.5,1) -- (0.5,2); 
 \draw [dotted] (0.5,2) --  (3.5,2);
 \draw (0.5,2) -- (1.5,1) -- (0.5,0);
 \draw [dotted] (1.5,1) --  (2.5,1);  
 \draw (3.5,2) -- (2.5,1) -- (3.5,0);
 \draw (0.5,2) -- (-0.5,3) -- (-1.5,2) --(-0.5,1);
 \draw (-4.5,2) -- (-3.5,3) -- (-2.5,2) -- (-3.5,1) --(-4.5,2); 
 \draw [dotted] (-3.5,1) -- (-0.5,1);
 \draw [dotted] (-1.5,2) -- (-2.5,2);
 \draw [dotted] (-0.5,3) -- (-3.5,3);
\draw [fill=white,white] (-0.5,1) circle [radius=0.2];
\draw [fill=white,white] (0.5,0) circle [radius=0.2];
\draw [fill=white,white] (4.5,1)  circle [radius=0.2];
\draw [fill=white,white] (3.5,0) circle [radius=0.3];
\draw [fill=white,white] (1.5,1) circle [radius=0.2];
\draw [fill=black,black] (1.5,1) circle [radius=0.05];
\draw [fill=white,white]  (0.5,2) circle [radius=0.2];
\draw [fill=white,white]  (3.5,2) circle [radius=0.2];
\draw [fill=white,white] (2.5,1) circle [radius=0.2];
\draw [fill=black,black] (2.5,1) circle [radius=0.05]; 
\draw [fill=white,white] (-1.5,2) circle [radius=0.2];
\draw [fill=black,black] (-1.5,2) circle [radius=0.05]; 
\draw [fill=white,white]  (-2.5,2) circle [radius=0.2];
\draw [fill=black,black]  (-2.5,2) circle [radius=0.05]; 
\draw [fill=white,white]   (-3.5,1) circle [radius=0.2];
\draw [fill=white,white]  (-3.5,3) circle [radius=0.2];
\draw [fill=white,white]  (-0.5,3)  circle [radius=0.2];
\draw [fill=white,white]   (-4.5,2)  circle [radius=0.2];
\draw [fill=white,white] (-0.5,1)  ellipse (20pt and 8pt);
\node at (-0.5,1) {\footnotesize $\tau_{\Lambda}(a_0)$};
\node at (0.5,0) {};
\node at (4.5,1) {\footnotesize $a_0$};
\node at (3.5,0) {};
\node at (-0.5,-1) {};
\node at (0.5,2) {\footnotesize $a_1$};
\node at (3.5,2) {};
\node at (1.5,1) {};
\node at (2.5,1) {};
\node at (-1.5,2) {};
\node at (-2.5,2) {};
\node at (-4.5,2) {\footnotesize $\tau_{\Lambda}(a_1)$};
\node at (-3.5,1) {};
\node at (-3.5,3) {\footnotesize $a_2$};
\node at (-0.5,3) {};
\end{tikzpicture} 
\caption{}

\end{figure}

 Continue in the same fashion to obtain a sequence of elements $a_0, \dots, a_k$ of $\Lambda$ such that $\tau_{\Lambda}(a_m) \in mesh_{\Lambda}(\tau_{\Lambda}(a_{m+1}),a_{m+1})$, $a_{m+1}\in mesh_{\Lambda}(\tau_{\Lambda}(a_{m}),a_{m})\setminus\{\tau_{\Lambda}(a_{m-1})\}$ for every $m = 0, \dots, k-1$ and $\tau_{\Lambda}(a_k)\not=-\infty$.
Now set $q_m := |mesh_{\Lambda}(\tau_{\Lambda}(a_{k}),a_k)|$ for $m = 0, \dots, k$.
 We claim that if the sequence $a_0, \dots, a_k$ is  maximal with respect to inclusion among sequences satisfying this property, then $q_k=1$.
 Indeed, if $q_k>1$, then there is a vertex $a_{k+1}$ with the smallest first coordinate among vertices of ${mesh_{\Lambda}(\tau_{\Lambda}(a_k),a_k)\setminus\{\tau_{\Lambda}(a_{k-1})\}}$.
 Because the sequence we consider is maximal, $\tau_{\Lambda}(a_{k+1})= (-\infty, p_2(a_{k+1}))$. Since $\theta (-\infty, p_2(a_{k+1}))$ is either $0$ or $-1$, we have $\theta(a_{k+1}) \geq \theta\tau_{\Lambda}(a_{k})$. Then $\theta(a_k) \geq \theta\tau_{\Lambda}(a_{k-1})$, since $$\theta(a_k) + \theta\tau_{\Lambda}(a_{k}) = \theta(a_{k+1}) + \theta\tau_{\Lambda}(a_{k-1}) + \sum\limits_{y \in mesh_{\Lambda}(\tau_{\Lambda}(a_{k}),a_k) \setminus \{a_{k+1}, \tau_{\Lambda}(a_{k-1})\}} \theta(y)$$
Continuing in the same way we get $\theta(a_1) \geq \theta\tau_{\Lambda}(a_{0})$. On the other hand, $\theta\tau_{\Lambda}(a_{0}) = \theta\tau_{\Lambda}(a_{0}) + \theta(a_0) > \theta(a_1)$, because $q_0 > 1$, a contradiction. 
 
 Take a sequence $a_0, \dots, a_k$ maximal with respect to inclusion and satisfying the properties described above with the smallest corresponding sequence of integers $(q_0, \dots, q_k)$ in the lexicographic order.
 Denote such a $(q_0, \dots, q_k)$ by $seq(\Lambda)$. If $seq(\Lambda) = (1)$, there is nothing to prove, as we remarked earlier. Now suppose that the statement is known for all $(\Lambda', \theta')$ with $seq(\Lambda') < seq(\Lambda)$. We will show that there exists a sequence of commutations and braidings that turns $(\Lambda, \theta)$ into $(\widetilde{\Lambda}, \widetilde{\theta})$ such that $seq(\widetilde{\Lambda}) < seq(\Lambda)$ in the lexicographic order. Observe that commutations  do not change $seq(\Lambda)$.

 Since $q_k = 1$, $mesh_{\Lambda}(\tau_{\Lambda}(a_{k}),a_k) = \{\tau_{\Lambda}(a_{k-1})\}$ and, possibly applying several commutations, one can assume that $\tau_{\Lambda}(a_{k}), \tau_{\Lambda}(a_{k-1}), a_k$ form a braid. Let $\tau_{\Lambda}(a_{k}) = (n,j)$, $\tau_{\Lambda}(a_{k-1}) = (n+1,l), a_k = (n+2,j)$. We already have $(n+1,t) \notin \Lambda$ for every $t \in N(j) \setminus \{l\}$.

 \begin{figure}[H]
  \centering 
 \begin{tikzpicture}
 \draw (-1.5,0) -- (-0.5,1) -- (0.5,0); 
 \draw [dotted] (0.5,0) -- (3.5,0);
 \draw (3.5,0) -- (4.5,1);
 \draw (-1.5,0) -- (-0.5,-1) --  (0.5,0);
 \draw (4.5,1) -- (3.5,2);
 \draw (-0.5,1) -- (0.5,2); 
 \draw [dotted] (0.5,2) --  (3.5,2);
 \draw (0.5,2) -- (1.5,1) -- (0.5,0);
 \draw [dotted] (1.5,1) --  (2.5,1);  
 \draw (3.5,2) -- (2.5,1) -- (3.5,0);
\draw [fill=white,white] (-1.5,0) circle [radius=0.2];
\draw [fill=white,white] (-0.5,1) circle [radius=0.2];
\draw [fill=white,white] (0.5,0) circle [radius=0.2];
\draw [fill=white,white] (4.5,1)  circle [radius=0.2];
\draw [fill=white,white] (3.5,0) circle [radius=0.3];
\draw [fill=white,white] (-0.5,-1) circle [radius=0.2];
\draw [fill=black,black] (-0.5,-1) circle [radius=0.05];
\draw [fill=white,white] (1.5,1) circle [radius=0.2];
\draw [fill=white,white]  (0.5,2) circle [radius=0.2];
\draw [fill=white,white]  (3.5,2) circle [radius=0.2];
\draw [fill=white,white] (2.5,1) circle [radius=0.2];
\draw [fill=black,black] (2.5,1) circle [radius=0.05]; 
\draw [fill=white,white] (3.5,0) ellipse (20pt and 8pt);
\node at (-1.5,0) {\footnotesize $\tau_{\Lambda}(a_{k})$};
\node at (-0.5,1) {\footnotesize $\tau_{\Lambda}(a_{k-1})$};
\node at (0.5,0) {\footnotesize $a_k$};
\node at (4.5,1) {\footnotesize $a_{k-1}$};
\node at (3.5,0) {\footnotesize $\tau_{\Lambda}(a_{k-2})$};
\node at (-0.5,-1) {};
\node at (0.5,2) {\footnotesize $z$};
\node at (3.5,2) {};
\node at (1.5,1) {};
\node at (2.5,1) {};
\node at (1.5,1) {\footnotesize $w$};
\end{tikzpicture} 
\caption{$\Lambda$}

\end{figure}

To perform a braiding on $\tau_{\Lambda}(a_{k}), \tau_{\Lambda}(a_{k-1}), a_k$, we also need to have $(n+3,j) \notin \Lambda$ and $(n+2,v) \notin \Lambda$ for every $v \in N(l) \setminus \{j\}$. To this end we first apply a sequence of commutations shifting all vertices $x$ of $\Lambda$ with $p_1(x) \geq n+2$, $x \neq a_k$, to the right. More precisely, the resulting set $\overline{\Lambda}$ is $\Lambda_1 \sqcup \Lambda_2$, where $\Lambda_1 = \{x \in \Lambda \mid p_1(x) \leq n+1 \} \cup \{a_k\}$, ${\Lambda_2 = \{(q+2, r) \mid (q,r) \in \Lambda \setminus \{a_k\}, q \geq n+2 \}}$ and $\overline{\theta}|_{\Lambda_1} = \theta|_{\Lambda_1}, \overline{\theta}(q+2,r) = \theta(q,r)$ for ${(q+2,r) \in \Lambda_2}$. It is clear that such a transformation can be obtained consequently applying commutations to all vertices of $\Lambda \setminus \Lambda_1$ starting with those with the largest first coordinate.

\begin{figure}[H]
    \centering
 
 \begin{tikzpicture}
 \draw (-1.5,0) -- (-0.5,1) -- (0.5,0); 
 \draw (-1.5,0) -- (-0.5,-1) --  (0.5,0);
  \draw (-0.5,1) -- (0.5,2); 
 \draw (0.5,2) -- (1.5,1) -- (0.5,0);
 \draw (1.5,1) -- (2.5,2);
 \draw (5.5,0) -- (6.5,1) -- (5.5,2);
 \draw [dotted] (5.5,2) -- (2.5,2); 
 \draw (2.5,2) -- (3.5,1) -- (2.5,0) -- (1.5,1);
 \draw [dotted] (2.5,0) -- (5.5,0);
 \draw [dotted] (4.5,1) -- (3.5,1);
 \draw (5.5,0) --  (4.5,1) -- (5.5,2);
\draw [fill=white,white] (-1.5,0) circle [radius=0.2];
\draw [fill=white,white] (-0.5,1) circle [radius=0.2];
\draw [fill=white,white] (0.5,0) circle [radius=0.2];
\draw [fill=white,white] (-0.5,-1) circle [radius=0.2];
\draw [fill=black,black] (-0.5,-1) circle [radius=0.05];
\draw [fill=white,white] (1.5,1) circle [radius=0.2];
\draw [fill=black,black] (1.5,1) circle [radius=0.05];
\draw [fill=white,white]  (0.5,2) circle [radius=0.2];
\draw [fill=white,white]  (2.5,2) circle [radius=0.2];
\draw [fill=black,black] (0.5,2) circle [radius=0.05];
\draw [fill=white,white] (6.5,1) circle [radius=0.2];
\draw [fill=white,white] (5.5,0) circle [radius=0.2];
\draw [fill=white,white] (5.5,2) circle [radius=0.2];
\draw [fill=white,white] (4.5,1) circle [radius=0.2];
\draw [fill=white,white] (3.5,1) circle [radius=0.2];
\draw [fill=white,white] (2.5,0) circle [radius=0.2];
\draw [fill=black,black] (2.5,0) circle [radius=0.05];
\draw [fill=white,white] (4.5,1) circle [radius=0.2];
\draw [fill=black,black] (4.5,1) circle [radius=0.05];
\draw [fill=white,white] (5.5,0) ellipse (20pt and 8pt);

\node at (-1.5,0) {\footnotesize $\tau_{\Lambda}(a_{k})$};
\node at (-0.5,1) {\footnotesize $\tau_{\Lambda}(a_{k-1})$};
\node at (0.5,0) {\footnotesize $a_k$};
\node at (-0.5,-1) {};
\node at (2.5,2) {\footnotesize $\overline{z}$};
\node at (1.5,1) {};
\node at (4.5,1) {};
\node at (5.5,2) {};
\node at (5.5,0) {\footnotesize $\overline{\tau_{\Lambda}(a_{k-2})}$};
\node at (6.5,1) {\footnotesize $\overline{a_{k-1}}$};
\node at (3.5,2) {};
\node at (3.5,1) {\footnotesize $\overline{w}$};

\end{tikzpicture} 

   \caption{$\overline{\Lambda}$}
\end{figure}
 
 \bigskip Now let $(\widetilde{\Lambda}, \widetilde{\theta})$ be a pair obtained by braiding $(\tau_{\Lambda}(a_{k}), \tau_{\Lambda}(a_{k-1}),a_k)$ in $\overline{\Lambda}$. Clearly, $q_0, \dots, q_{k-2}$ remain unchanged. However, the mesh ending at $a_{k-1}$ now starts at $\tau_{\widetilde{\Lambda}}(a_{k-1})=(n+3,l)$, and hence obviously does not contain the vertex $a_k=(n+2,j)$.  We see that $|mesh_{\widetilde{\Lambda}}(\tau_{\widetilde{\Lambda}}(a_{k-1}), a_{k-1})| = q_{k-1} - 1$.
 Since any sequence that starts with $q_0, \dots, q_{k-1} -1$ is smaller in the lexicographic order than the sequence $(q_1, \dots, q_{k-1}, q_k)$, we have $seq(\widetilde{\Lambda}) < seq(\Lambda)$.

\begin{figure}[H]
\centering
 \begin{tikzpicture}
 \draw (-1.5,0) -- (-0.5,1) -- (0.5,0); 
 \draw (-1.5,0) -- (-0.5,-1) --  (0.5,0);
  \draw (-0.5,1) -- (0.5,2); 
 \draw (0.5,2) -- (1.5,1) -- (0.5,0);
 \draw (1.5,1) -- (2.5,2);
 \draw (5.5,0) -- (6.5,1) -- (5.5,2);
 \draw [dotted] (5.5,2) -- (2.5,2); 
 \draw (2.5,2) -- (3.5,1) -- (2.5,0) -- (1.5,1);
 \draw [dotted] (2.5,0) -- (5.5,0);
 \draw [dotted] (4.5,1) -- (3.5,1);
 \draw (5.5,0) --  (4.5,1) -- (5.5,2);
\draw [fill=white,white] (-1.5,0) circle [radius=0.2];
\draw [fill=white,white] (-0.5,1) circle [radius=0.2];
\draw [fill=black,black] (-1.5,0) circle [radius=0.05];
\draw [fill=white,white] (0.5,0) circle [radius=0.2];
\draw [fill=white,white] (-0.5,-1) circle [radius=0.2];
\draw [fill=black,black] (-0.5,-1) circle [radius=0.05];
\draw [fill=white,white] (1.5,1) circle [radius=0.2];
\draw [fill=white,white]  (0.5,2) circle [radius=0.2];
\draw [fill=white,white]  (2.5,2) circle [radius=0.2];
\draw [fill=black,black] (0.5,2) circle [radius=0.05];
\draw [fill=white,white] (6.5,1) circle [radius=0.2];
\draw [fill=white,white] (5.5,0) circle [radius=0.2];
\draw [fill=white,white] (5.5,2) circle [radius=0.2];
\draw [fill=white,white] (4.5,1) circle [radius=0.2];
\draw [fill=white,white] (3.5,1) circle [radius=0.2];
\draw [fill=white,white] (2.5,0) circle [radius=0.2];
\draw [fill=black,black] (2.5,0) circle [radius=0.05];
\draw [fill=white,white] (4.5,1) circle [radius=0.2];
\draw [fill=black,black] (4.5,1) circle [radius=0.05];
\draw [fill=white,white] (5.5,0) ellipse (20pt and 8pt);

\node at (-1.5,0) {};
\node at (-0.5,1) {\footnotesize $\tau_{{\Lambda}}(a_{k-1})$};
\node at (0.5,0) {\footnotesize $a_k$};
\node at (-0.5,-1) {};
\node at (2.5,2) {\footnotesize $\overline{z}$};
\node at (1.5,1) {\footnotesize $\tau_{\widetilde{\Lambda}}(a_{k-1})$};
\node at (4.5,1) {};
\node at (5.5,2) {};
\node at (5.5,0) {\footnotesize $\overline{\tau_{\Lambda}(a_{k-2})}$};
\node at (6.5,1) {\footnotesize $\overline{a_{k-1}}$};
\node at (3.5,2) {};
\node at (3.5,1) {\footnotesize $\overline{w}$};

\end{tikzpicture} 
   \caption{$\widetilde{\Lambda}$}
\end{figure}

 \end{proof}
 
   {\bf Example.} Consider $\Lambda_\gamma$ as in the previous example.

 \begin{figure}[H]
 \noindent
 \begin{minipage}{0.4\textwidth}
\centering 

\begin{tikzpicture}
\draw (-3,1.5) -- (-2,0.5) -- (-1,1.5) -- (0,0.5) --(1,1.5);
\draw  (-3,0.5) --  (1,0.5);
\draw (-3,-0.5) --  (-2,0.5) --  (-1,-0.5) --(0,0.5); 
\draw  (-4,0.5) --(-3,1.5); 
\draw  (-4,0.5) -- (-3,0.5); 
\draw  (-4,0.5) -- (-3,-0.5); 
\draw (0,0.5) --  (1,-0.5);

\draw [fill=white,white]  (-3,-0.5) circle [radius=0.2];
\draw [fill=white,white]  (-3,0.5) circle [radius=0.2];
\draw [fill=white,white]  (-3,1.5) circle [radius=0.2];
\draw [fill=white,white]  (-2,0.5) circle [radius=0.2];
\draw [fill=white,white]  (-1,-0.5) circle [radius=0.2];
\draw [fill=white,white]  (-1,0.5) circle [radius=0.2];
\draw [fill=white,white]  (-1,1.5) circle [radius=0.2];
\draw [fill=white,white]  (0,0.5) circle [radius=0.2];
\draw [fill=white,white]  (0,0.5) circle [radius=0.2];
\draw [fill=white,white] (1,-0.5) circle [radius=0.2];
\draw [fill=white,white]   (1,0.5)  circle [radius=0.2];
\draw [fill=white,white]   (1,1.5)  circle [radius=0.2];
\draw [fill=white,white]  (-4,0.5)  circle [radius=0.2];
\draw [fill=black,black]   (1,0.5)  circle [radius=0.05];
\draw [fill=black,black]  (1,-0.5) circle [radius=0.05];

\node at (-3,-0.5)  {\footnotesize 1};
\node at (-3,0.5)  {\footnotesize 1};
\node at (-3,1.5)  {\footnotesize 1};
\node at (-2,0.5)  {\footnotesize 2};

\node at (-1,-0.5)  {\footnotesize 1};
\node at (-1,0.5)  {\footnotesize 1};
\node at (-1,1.5)  {\footnotesize 1};
\node at (0,0.5)  {\footnotesize 1};
\node at (1,-0.5) {};
\node at (1,0.5) {};
\node at (1,1.5)  {\footnotesize 0};

\node at (-4,0.5) {\footnotesize 1};
\end{tikzpicture}
\caption{The first application of Lemma 16 requires just one braiding.}

\end{minipage}
\hspace{1.3cm}
\begin{minipage}[t]{0.4\textwidth}
\raggedleft
\begin{tikzpicture}
\draw (-3,1.5) -- (-2,0.5) -- (-1,1.5) -- (0,0.5) --(1,1.5);
\draw  (-3,0.5) --  (1,0.5);
\draw (-3,-0.5) --  (-2,0.5) --  (-1,-0.5) --(0,0.5); 
\draw  (-4,0.5) --(-3,1.5); 
\draw  (-4,0.5) -- (-3,0.5); 
\draw  (-4,0.5) -- (-3,-0.5); 
\draw (0,0.5) --  (1,-0.5);
\draw (1,1.5) -- (2,0.5) -- (1,0.5);
\draw (2,0.5) --  (1,-0.5);

\draw [fill=white,white]  (-3,-0.5) circle [radius=0.2];
\draw [fill=white,white]  (-3,0.5) circle [radius=0.2];
\draw [fill=white,white]  (-3,1.5) circle [radius=0.2];
\draw [fill=white,white]  (-2,0.5) circle [radius=0.2];
\draw [fill=white,white]  (-1,-0.5) circle [radius=0.2];
\draw [fill=white,white]  (-1,0.5) circle [radius=0.2];
\draw [fill=white,white]  (-1,1.5) circle [radius=0.2];
\draw [fill=white,white]  (0,0.5) circle [radius=0.2];
\draw [fill=white,white]  (0,0.5) circle [radius=0.2];
\draw [fill=white,white] (1,-0.5) circle [radius=0.2];
\draw [fill=white,white]   (1,0.5)  circle [radius=0.2];
\draw [fill=white,white]   (1,1.5)  circle [radius=0.2];
\draw [fill=white,white]  (-4,0.5)  circle [radius=0.2];
\draw [fill=white,white]  (2,0.5)  circle [radius=0.2];
\draw [fill=black,black]   (1,0.5)  circle [radius=0.05];
\draw [fill=black,black]  (1,-0.5) circle [radius=0.05];
\draw [fill=black,black]  (-1,1.5)  circle [radius=0.05];

\node at (-3,-0.5)  {\footnotesize 1};
\node at (-3,0.5)  {\footnotesize 1};
\node at (-3,1.5)  {\footnotesize 1};
\node at (-2,0.5)  {\footnotesize 2};

\node at (-1,-0.5)  {\footnotesize 1};
\node at (-1,0.5)  {\footnotesize 1};
\node at (-1,1.5) {} ;
\node at (0,0.5)  {\footnotesize 0};
\node at (1,-0.5) {};
\node at (1,0.5) {};
\node at (1,1.5)  {\footnotesize 1};

\node at (-4,0.5) {\footnotesize 1};
\node at (2,0.5) {\footnotesize 1};

\end{tikzpicture}
\caption{Now the sequence of meshes starting at $0$ is $(2,1)$.}

\end{minipage}
\end{figure} 

\begin{figure}[H]
\centering 
\begin{tikzpicture}
\draw (-3,1.5) -- (-2,0.5) -- (-1,1.5) -- (0,0.5) --(1,1.5);
\draw  (-3,0.5) --  (1,0.5);
\draw (-3,-0.5) --  (-2,0.5) --  (-1,-0.5) --(0,0.5); 
\draw  (-4,0.5) --(-3,1.5); 
\draw  (-4,0.5) -- (-3,0.5); 
\draw  (-4,0.5) -- (-3,-0.5); 
\draw (0,0.5) --  (1,-0.5);
\draw (1,1.5) -- (2,0.5) -- (1,0.5);
\draw (2,0.5) --  (1,-0.5);
\draw (2,0.5) --  (3,1.5) --(4,0.5) --  (2,0.5) -- (3,-0.5) --  (4,0.5);

\draw [fill=white,white]  (3,1.5) circle [radius=0.2];
\draw [fill=white,white]   (3,0.5) circle [radius=0.2];
\draw [fill=white,white]  (3,-0.5)  circle [radius=0.2];
\draw [fill=white,white]  (4,0.5) circle [radius=0.2];
\draw [fill=white,white]  (-3,-0.5) circle [radius=0.2];
\draw [fill=white,white]  (-3,0.5) circle [radius=0.2];
\draw [fill=white,white]  (-3,1.5) circle [radius=0.2];
\draw [fill=white,white]  (-2,0.5) circle [radius=0.2];
\draw [fill=white,white]  (-1,-0.5) circle [radius=0.2];
\draw [fill=white,white]  (-1,0.5) circle [radius=0.2];
\draw [fill=white,white]  (-1,1.5) circle [radius=0.2];
\draw [fill=white,white]  (0,0.5) circle [radius=0.2];
\draw [fill=white,white]  (0,0.5) circle [radius=0.2];
\draw [fill=white,white] (1,-0.5) circle [radius=0.2];
\draw [fill=white,white]   (1,0.5)  circle [radius=0.2];
\draw [fill=white,white]   (1,1.5)  circle [radius=0.2];
\draw [fill=white,white]  (-4,0.5)  circle [radius=0.2];
\draw [fill=white,white]  (2,0.5)  circle [radius=0.2];
\draw [fill=black,black]  (3,-0.5)  circle [radius=0.05];
\draw [fill=black,black]  (1,1.5)   circle [radius=0.05];
\draw [fill=black,black]  (1,-0.5) circle [radius=0.05];
\draw [fill=black,black]  (0,0.5)  circle [radius=0.05];
\draw [fill=black,black]  (-1,1.5)  circle [radius=0.05];
\draw [fill=black,black] (-1,0.5)  circle [radius=0.05];
\draw [fill=black,black]    (3,0.5) circle [radius=0.05];

\node at (-3,-0.5)  {\footnotesize 1};
\node at (-3,0.5)  {\footnotesize 1};
\node at (-3,1.5)  {\footnotesize 1};
\node at (-2,0.5)  {\footnotesize 2};

\node at (-1,-0.5)  {\footnotesize 1};
\node at (-1,0.5)  {};
\node at (-1,1.5) {} ;
\node at (0,0.5)  {};
\node at (1,-0.5) {};
\node at (1,0.5) {\footnotesize 1};
\node at (1,1.5)  {};

\node at (-4,0.5) {\footnotesize 1};
\node at (2,0.5) {\footnotesize 0};

\node at (3,1.5) {\footnotesize 1};
\node at (3,0.5) {};
\node at (3,-0.5) {};
\node at (4,0.5) {\footnotesize 1};

\end{tikzpicture}
\caption{The second application of Lemma \ref{mesh} begins with a sequence of commutations.}
\end{figure} 

\begin{figure}[H]
\centering 
\begin{tikzpicture}
\draw (-3,1.5) -- (-2,0.5) -- (-1,1.5) -- (0,0.5) --(1,1.5);
\draw  (-3,0.5) --  (1,0.5);
\draw (-3,-0.5) --  (-2,0.5) --  (-1,-0.5) --(0,0.5); 
\draw  (-4,0.5) --(-3,1.5); 
\draw  (-4,0.5) -- (-3,0.5); 
\draw  (-4,0.5) -- (-3,-0.5); 
\draw (0,0.5) --  (1,-0.5);
\draw (1,1.5) -- (2,0.5) -- (1,0.5);
\draw (2,0.5) --  (1,-0.5);
\draw (2,0.5) --  (3,1.5) --(4,0.5) --  (2,0.5) -- (3,-0.5) --  (4,0.5);

\draw [fill=white,white]  (3,1.5) circle [radius=0.2];
\draw [fill=white,white]   (3,0.5) circle [radius=0.2];
\draw [fill=white,white]  (3,-0.5)  circle [radius=0.2];
\draw [fill=white,white]  (4,0.5) circle [radius=0.2];
\draw [fill=white,white]  (-3,-0.5) circle [radius=0.2];
\draw [fill=white,white]  (-3,0.5) circle [radius=0.2];
\draw [fill=white,white]  (-3,1.5) circle [radius=0.2];
\draw [fill=white,white]  (-2,0.5) circle [radius=0.2];
\draw [fill=white,white]  (-1,-0.5) circle [radius=0.2];
\draw [fill=white,white]  (-1,0.5) circle [radius=0.2];
\draw [fill=white,white]  (-1,1.5) circle [radius=0.2];
\draw [fill=white,white]  (0,0.5) circle [radius=0.2];
\draw [fill=white,white]  (0,0.5) circle [radius=0.2];
\draw [fill=white,white] (1,-0.5) circle [radius=0.2];
\draw [fill=white,white]   (1,0.5)  circle [radius=0.2];
\draw [fill=white,white]   (1,1.5)  circle [radius=0.2];
\draw [fill=white,white]  (-4,0.5)  circle [radius=0.2];
\draw [fill=white,white]  (2,0.5)  circle [radius=0.2];
\draw [fill=black,black]  (3,-0.5)  circle [radius=0.05];
\draw [fill=black,black]  (1,1.5)   circle [radius=0.05];
\draw [fill=black,black]  (1,-0.5) circle [radius=0.05];
\draw [fill=black,black]  (-1,1.5)  circle [radius=0.05];
\draw [fill=black,black] (-1,0.5)  circle [radius=0.05];
\draw [fill=black,black]   (3,0.5) circle [radius=0.05];
\draw [fill=black,black]   (-3,-0.5) circle [radius=0.05];

\node at (-3,-0.5)  {};
\node at (-3,0.5)  {\footnotesize 1};
\node at (-3,1.5)  {\footnotesize 1};
\node at (-2,0.5)  {\footnotesize 1};

\node at (-1,-0.5)  {\footnotesize 2};
\node at (-1,0.5)  {};
\node at (-1,1.5) {} ;
\node at (0,0.5)  {\footnotesize 1};
\node at (1,-0.5) {};
\node at (1,0.5) {\footnotesize 1};
\node at (1,1.5)  {};

\node at (-4,0.5) {\footnotesize 1};
\node at (2,0.5) {\footnotesize 0};

\node at (3,1.5) {\footnotesize 1};
\node at (3,0.5) {};
\node at (3,-0.5) {};
\node at (4,0.5) {\footnotesize 1};

\end{tikzpicture}
\caption{Now three vertices can be braided and the sequence of meshes starting at $0$ is $(1)$ again.}
\end{figure} 

\begin{figure}[H]
\centering 
\begin{tikzpicture}
\draw (-3,1.5) -- (-2,0.5) -- (-1,1.5) -- (0,0.5) --(1,1.5);
\draw  (-3,0.5) --  (1,0.5);
\draw (-3,-0.5) --  (-2,0.5) --  (-1,-0.5) --(0,0.5); 
\draw  (-4,0.5) --(-3,1.5); 
\draw  (-4,0.5) -- (-3,0.5); 
\draw  (-4,0.5) -- (-3,-0.5); 
\draw (0,0.5) --  (1,-0.5);
\draw (1,1.5) -- (2,0.5) -- (1,0.5);
\draw (2,0.5) --  (1,-0.5);
\draw (2,0.5) --  (3,1.5) --(4,0.5) --  (2,0.5) -- (3,-0.5) --  (4,0.5);

\draw [fill=white,white]  (3,1.5) circle [radius=0.2];
\draw [fill=white,white]   (3,0.5) circle [radius=0.2];
\draw [fill=white,white]  (3,-0.5)  circle [radius=0.2];
\draw [fill=white,white]  (4,0.5) circle [radius=0.2];
\draw [fill=white,white]  (-3,-0.5) circle [radius=0.2];
\draw [fill=white,white]  (-3,0.5) circle [radius=0.2];
\draw [fill=white,white]  (-3,1.5) circle [radius=0.2];
\draw [fill=white,white]  (-2,0.5) circle [radius=0.2];
\draw [fill=white,white]  (-1,-0.5) circle [radius=0.2];
\draw [fill=white,white]  (-1,0.5) circle [radius=0.2];
\draw [fill=white,white]  (-1,1.5) circle [radius=0.2];
\draw [fill=white,white]  (0,0.5) circle [radius=0.2];
\draw [fill=white,white]  (0,0.5) circle [radius=0.2];
\draw [fill=white,white] (1,-0.5) circle [radius=0.2];
\draw [fill=white,white]   (1,0.5)  circle [radius=0.2];
\draw [fill=white,white]   (1,1.5)  circle [radius=0.2];
\draw [fill=white,white]  (-4,0.5)  circle [radius=0.2];
\draw [fill=white,white]  (2,0.5)  circle [radius=0.2];
\draw [fill=black,black]  (3,-0.5)  circle [radius=0.05];
\draw [fill=black,black]  (1,1.5)   circle [radius=0.05];
\draw [fill=black,black]  (1,-0.5) circle [radius=0.05];
\draw [fill=black,black]  (-1,1.5)  circle [radius=0.05];
\draw [fill=black,black] (2,0.5)  circle [radius=0.05];
\draw [fill=black,black]   (3,0.5) circle [radius=0.05];
\draw [fill=black,black]   (-3,-0.5) circle [radius=0.05];

\node at (-3,-0.5)  {};
\node at (-3,0.5)  {\footnotesize 1};
\node at (-3,1.5)  {\footnotesize 1};
\node at (-2,0.5)  {\footnotesize 1};

\node at (-1,-0.5)  {\footnotesize 2};
\node at (-1,0.5)  {\footnotesize 0};
\node at (-1,1.5) {} ;
\node at (0,0.5)  {\footnotesize 1};
\node at (1,-0.5) {};
\node at (1,0.5) {\footnotesize 1};
\node at (1,1.5)  {};

\node at (-4,0.5) {\footnotesize 1};
\node at (2,0.5) {};

\node at (3,1.5) {\footnotesize 1};
\node at (3,0.5) {};
\node at (3,-0.5) {};
\node at (4,0.5) {\footnotesize 1};

\end{tikzpicture}
\caption{Braiding the remaining three vertices in the sequence of meshes finishes the second application of Lemma \ref{mesh}. }
\end{figure} 

\begin{remark} One can prove the assertion converse to Theorem \ref{mesh_thm}. Namely, if $(\Lambda, \theta)$ satisfying mesh relations is such that $\theta$ is non-negative on $\Lambda$, $\theta(-\infty,i) = -1$ for some $i \in \Gamma_0$, $\theta(-\infty,k) = 0$ for all $k \in \Gamma_0\setminus \{i\}$ and $\gamma_{\Lambda}$ is left-divisible by some $s_j$ in $B_\Gamma^+$ with $j\neq i$, then $\theta$ vanishes on some vertex of $\Lambda$. Indeed, in this case $(\Lambda,\theta)$ can be transformed using commutations and braidings into $(\widetilde\Lambda,\widetilde\theta)$ such that there is a vertex $x=(n,j)\in\widetilde\Lambda$ with $\tau_{\widetilde\Lambda}(x)=(-\infty,j)$ and $mesh_{\widetilde\Lambda}\big(\tau_{\widetilde\Lambda}(x),x\big)=\varnothing$. Then the mesh relation corresponding to $x$ implies $\widetilde\theta(x)=0$ and the claim follows from the fact that commutations and braidings do not change the multiset of values of $\theta$.
\end{remark}

    \section{An application to derived Picard groups}\label{trpic}
  For this section we assume that the field $\kk$ is algebraically closed. All modules in this section are assumed to be left unless explicitly stated otherwise. Let $\mathfrak{D}_{\Gamma} = D^b(\modl-\Lambda_{\Gamma})$, the bounded derived category of finitely generated $\Lambda_{\Gamma}$-modules, where $\Gamma$ is a simply laced Dynkin diagram and $\Lambda_{\Gamma}$ is the trivial extension algebra of the $\Gamma$ diagram with alternating orientation. In other words, $\Lambda_\Gamma = kQ_\Gamma/I_\Gamma$, where $Q_\Gamma$ is one of the following quivers:

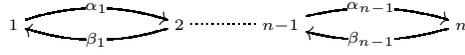
\begin{figure}[h!]
\centering
{\small
\begin{equation*}
\xymatrix{\scriptstyle{1} \ar@/^/[rr]|{\alpha_1}
&&\scriptstyle{2} \ar@/^/[ll]|{\beta_1}\ar@{.}[r]
&\scriptstyle{n-1} \ar@/^/[rr]|{\alpha_{n-1}}
&&\scriptstyle{n}\ar@/^/[ll]|{\beta_{n-1}}
}
\end{equation*} }
\caption{$\Gamma = A_n$}
\end{figure}

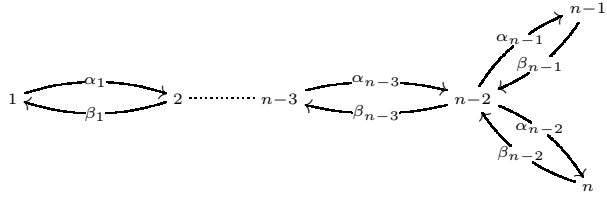
\begin{figure}[h!]
\centering
{\small
\begin{equation*}
\xymatrix{%
&& & && &\scriptstyle{n-1}\ar@/^/[dl]|{\beta_{n-1}} \\
\scriptstyle{1} \ar@/^/[rr]|{\alpha_1}
&&\scriptstyle{2} \ar@/^/[ll]|{\beta_1}\ar@{.}[r]
&\scriptstyle{n-3} \ar@/^/[rr]|{\alpha_{n-3}}
&&\scriptstyle{n-2}\ar@/^/[ll]|{\beta_{n-3}}\ar@/^/[ur]|{\alpha_{n-1	}} \ar@/^/[dr]|{\alpha_{n-2}} \\
&& & && &\scriptstyle{n}\ar@/^/[ul]|{\beta_{n-2}}
}
\end{equation*} }
\caption{$\Gamma = D_n$}
\end{figure}

\begin{figure}[H]
\centering
{\tiny
\begin{equation*}
\xymatrix{%
&& && \scriptstyle{1} \ar@/^/[dd]|{\alpha_1} \\
&& && \\
\scriptstyle{4} \ar@/^/[rr]|{\alpha_4}
&&\scriptstyle{3} \ar@/^/[ll]|{\beta_4}\ar@/^/[rr]|{\alpha_3}
&&\scriptstyle{2} \ar@/^/[ll]|{\beta_3}\ar@/^/[rr]|{\alpha_{2}}\ar@/^/[uu]|{\beta_1}
&&\scriptstyle{5}\ar@/^/[ll]|{\beta_{2}}\ar@{.}[r]
&\scriptstyle{n-1} \ar@/^/[rr]|{\alpha_{n-1}} 
&&\scriptstyle{n} \ar@/^/[ll]|{\beta_{n-1}} 
}
\end{equation*} }
\caption{$\Gamma = E_n$ $(n=6,7,8)$}
\end{figure}
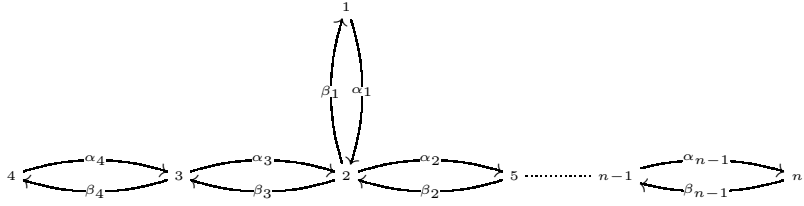

The ideal $I_\Gamma$ of $kQ_\Gamma$ is generated by all paths of length greater or equal to $2$, except for the paths of length $2$ starting and ending at the same vertex, and the differences of any two paths of length two starting at the same vertex.

The $\kk$-algebra $\Lambda_\Gamma$ is finite-dimensional and symmetric. Let $e_i$ be the idempotent associated with the vertex $i$ of the quiver $Q_\Gamma$. Denote by $P_i = \Lambda_\Gamma e_i$ the corresponding indecomposable projective module. Note that $P_i$ is a 0-spherical object of $\mathfrak{D}_\Gamma$. Indeed, the first condition is satisfied automatically. Since $P_i$ is projective, $\Ext_{\Lambda_\Gamma}^m(P_i, -)$ vanishes for every $m \neq 0$, and hence the second condition simply means that $\End_{\Lambda_\Gamma}(P_i) \cong \kk[t]/(t^2)$. Finally, since $\Lambda_\Gamma$ is symmetric, the last condition is satisfied automatically as well due to an isomorphism of functors $\Hom(P_i, -) \cong \Hom(-,P_i)^*$. To sum up, it is now clear that $\{P_i\}_{i \in {(Q_\Gamma)}_0}$ is a $\Gamma$-configuration of $0$-spherical objects of $\mathfrak{D}_\Gamma$.

 \begin{remark} In this less general setting we could have defined spherical twists in the following way:

 \begin{definition} (Grant, \cite{G}) The spherical twist functor along $P_i$ is
$$t_{P_i} \colon \mathfrak{D}_\Gamma \to \mathfrak{D}_\Gamma$$
$$t_{P_i}(-) = cone(\Lambda_{\Gamma}e_i \otimes_\kk e_i\Lambda_{\Gamma} \xrightarrow{m} \Lambda_{\Gamma}) \otimes_{\Lambda_{\Gamma}}^L -$$ where $m$ is the multiplication map $m(ae_i \otimes e_ib) = ab$ of $\Lambda_\Gamma$-$\Lambda_\Gamma$ bimodules.
 \end{definition}

 On objects it is easy to check that it coincides with our original definition, which is
 $$t_{P_i}(X) = cone(P_i \otimes_\kk \Hom(P_i, X) \xrightarrow{ev} X)$$
with $ev$ the obvious evaluation map. 
 Indeed, since $cone(\Lambda_{\Gamma}e_i \otimes_\kk e_i \Lambda_{\Gamma} \xrightarrow{m}  \Lambda_{\Gamma})$ is a two-term complex of right-projective bimodules,
$$cone(\Lambda_{\Gamma}e_i \otimes_\kk e_i\Lambda_{\Gamma} \xrightarrow{m} \Lambda_{\Gamma}) \otimes_{\Lambda_{\Gamma}}^L X \cong cone\big((\Lambda_{\Gamma}e_i \otimes_\kk e_i\Lambda_{\Gamma} \otimes_{\Lambda_{\Gamma}} X) \xrightarrow{} (\Lambda_{\Gamma} \otimes_{\Lambda_{\Gamma}} X)\big) $$
 $$\cong cone\big(P_i \otimes_\kk (e_i\Lambda_{\Gamma} \otimes_{\Lambda_{\Gamma}} X) \xrightarrow{} X\big)  \cong cone\big(P_i \otimes_\kk \Hom(P_i, X) \xrightarrow{ev} X\big).$$
 \end{remark} 
 
 \begin{definition} (R. Rouquier, A. Zimmermann, \cite{RZ}) Let $A$ be a finite-dimensional algebra. The {\it derived Picard group} $TrPic(A)$ of $A$ is the group of isomorphism classes of objects of the derived category of $A \otimes A^{\rm op}$-modules, invertible under  $\otimes_A^L$. Equivalently, $TrPic(A)$ is the group of standard autoequivalences of $D^b(\modl-A)$ modulo natural isomorphisms.   \end{definition}

It is known (see, for example, \cite{VZ1}) that in the case $\Gamma=A_n$ the subgroup of $TrPic(\Lambda_{\Gamma})$  generated by the spherical twist functors $t_{P_i}$ is isomorphic to the braid group $B_{A_n}$ on $(n+1)$ strands. The next corollary of Theorem 1 generalizes this result.

 \begin{corollary} The subgroup of the derived Picard group $TrPic(\Lambda_{\Gamma})$ of $\Lambda_{\Gamma}$ generated by the spherical twist functors $t_{P_i}$ is isomorphic to the generalized braid group $B_{\Gamma}$ of type $\Gamma$. \end{corollary}

 \end{document}